\documentclass[proceedings,submission]{dmtcs}

\usepackage[utf8x]{inputenc}
\usepackage[english]{babel}
\usepackage{amsmath, amsfonts, amssymb}
\usepackage{shuffle}
\usepackage{enumitem}
\usepackage{color}
\usepackage{tikz}
\usetikzlibrary{decorations.pathmorphing}

\author{Samuele Giraudo\addressmark{1}}
\title{Algebraic and combinatorial structures on Baxter permutations}
\address{\addressmark{1}Institut Gaspard Monge, Université Paris-Est
    Marne-la-Vallée, 5 Boulevard Descartes, Champs-sur-Marne,
    77454 Marne-la-Vallée cedex 2, France
}

\newtheorem{Theoreme}{Theorem}[section]
\newtheorem{Proposition}[Theoreme]{Proposition}
\newtheorem{Lemme}[Theoreme]{Lemma}
\newtheorem{Definition}[Theoreme]{Definition}
\newtheorem{Remarque}[Theoreme]{Remark}

\newcommand{\EnsPermu}{\mathfrak{S}}
\newcommand{\EnsRel}{\mathbb{Z}}
\newcommand{\EnsAB}{\mathcal{B}\mathcal{T}}
\newcommand{\EnsABJ}{\mathcal{T}\mathcal{B}\mathcal{T}}
\newcommand{\EnsPermuBX}{\mathfrak{S}^{\operatorname{B}}}

\newcommand{\EquivBX}{{\hspace{.2em} \equiv_{\operatorname{B}} \hspace{.2em}}}
\newcommand{\EquivS}{{\hspace{.2em} \equiv_{\operatorname{S}} \hspace{.2em}}}
\newcommand{\EquivCoS}{{\hspace{.2em} \equiv_{\operatorname{S}^\#} \hspace{.2em}}}
\newcommand{\AdjBXA}{{\hspace{.2em} \leftrightharpoons_{\operatorname{B}} \hspace{.2em}}}
\newcommand{\AdjBXB}{{\hspace{.2em} \rightleftharpoons_{\operatorname{B}} \hspace{.2em}}}
\newcommand{\AdjS}{{\hspace{.2em} \leftrightharpoons_{\operatorname{S}} \hspace{.2em}}}
\newcommand{\AdjCoS}{{\hspace{.2em} \leftrightharpoons_{\operatorname{S}^\#} \hspace{.2em}}}
\newcommand{\PSymb}{{\mathbb{P}}}
\newcommand{\QSymb}{{\mathbb{Q}}}
\newcommand{\FQSym}{{\bf FQSym}}
\newcommand{\Baxter}{{\bf Baxter}}
\newcommand{\PBT}{{\bf PBT}}
\newcommand{\Sym}{{\bf Sym}}
\newcommand{\FSym}{{\bf FSym}}
\newcommand{\Bell}{{\bf Bell}}
\newcommand{\F}{{\bf F}}
\newcommand{\PP}{{\bf P}}
\newcommand{\E}{{\bf E}}
\newcommand{\HH}{{\bf H}}
\newcommand{\Over}{\hspace{.1em} \diagup \hspace{.1em}}
\newcommand{\Under}{\hspace{.1em} \diagdown \hspace{.1em}}
\newcommand{\Prod}{\cdot}
\newcommand{\Gauche}{\prec}
\newcommand{\Droite}{\succ}
\newcommand{\DeltaG}{\Delta_\Gauche}
\newcommand{\DeltaD}{\Delta_\Droite}

\newcommand{\Tenseur}{\otimes}
\newcommand{\Std}{\operatorname{std}}
\newcommand{\ArbreVide}{\perp}
\newcommand{\Canop}{\operatorname{cnp}}
\newcommand{\Sloane}[1]{{\bf #1}}
\newcommand{\Eval}{\operatorname{eval}}
\newcommand{\OrdPermu}{\leq_{\operatorname{P}}}
\newcommand{\OrdTam}{\leq_{\operatorname{T}}}
\newcommand{\OrdBX}{\leq_{\operatorname{B}}}

\newcommand{\K}{\mathbb{K}}
\newcommand{\LettreA}{{\tt a}}
\newcommand{\LettreB}{{\tt b}}
\newcommand{\LettreC}{{\tt c}}
\newcommand{\LettreD}{{\tt d}}

\definecolor{Noir}{RGB}{0,0,0}
\definecolor{Rouge}{RGB}{205,35,38}
\definecolor{Bleu}{RGB}{2,60,195}
\definecolor{Vert}{RGB}{23,103,1}
\definecolor{Orange}{RGB}{255,113,15}
\definecolor{Blanc}{RGB}{255,255,255}

\tikzstyle{Noeud} = [circle, draw = Bleu!100, fill = Bleu!50, thick, inner sep = 0pt, minimum size = 10mm]
\tikzstyle{Feuille} = [rectangle, draw = Noir!100, fill = Noir!30, thick, inner sep = 0pt, minimum size = 2mm]
\tikzstyle{Arete} = [Rouge!80, thick, draw, line width = 2pt]
\tikzstyle{SArbre} = [rectangle, draw = Orange!100, fill = Orange!30, thick, inner sep = 0pt, minimum size = 8mm, font = \Huge]
\tikzstyle{Marque1} = [draw = Vert!100, fill = Vert!60]
\tikzstyle{Marque2} = [draw = Orange!100, fill = Orange!90]
\tikzstyle{EtiqClair} = [draw = Noir!100, fill = Blanc!100, font = \Huge]
\tikzstyle{EtiqFonce} = [draw = Bleu!100, fill = Bleu!20, font = \Huge]

\begin{document}

\maketitle

\begin{abstract}
    \paragraph{Abstract.}
    We give a new construction of a Hopf subalgebra of the Hopf algebra
    of Free quasi-symmetric functions whose bases are indexed by objects
    belonging to the Baxter combinatorial family (\emph{i.e.} Baxter permutations,
    pairs of twin binary trees, \emph{etc.}). This construction relies on
    the definition of the Baxter monoid, analog of the plactic monoid and
    the sylvester monoid, and on a Robinson-Schensted-like insertion algorithm.
    The algebraic properties of this Hopf algebra are studied. This Hopf
    algebra appeared for the first time in the work of Reading [Lattice
    congruences, fans and Hopf algebras, {\it Journal of Combinatorial
    Theory Series A}, 110:237--273, 2005].

    \paragraph{Résumé.}
    Nous proposons une nouvelle construction d'une sous-algèbre de Hopf
    de l'algèbre de Hopf des fonctions quasi-symétriques libres dont
    les bases sont indexées par les objets de la famille combinatoire
    de Baxter (\emph{i.e.} permutations de Baxter, couples d'arbres binaires
    jumeaux, \emph{etc.}). Cette construction repose sur la définition
    du monoïde de Baxter, analogue du monoïde plaxique et du monoïde
    sylvestre, et d'un algorithme d'insertion analogue à l'algorithme de
    Robinson-Schensted. Les propriétés algébriques de cette algèbre de Hopf
    sont étudiées. Cette algèbre de Hopf est apparue pour la première fois
    dans le travail de Reading [Lattice congruences, fans  and Hopf algebras,
    {\it Journal of Combinatorial Theory Series A}, 110:237--273, 2005].
\end{abstract}

\section{Introduction} \label{sec:Intro}
In the recent years, many combinatorial Hopf algebras, whose bases are
indexed by combinatorial objects, have been intensively studied. For example,
the Malvenuto-Reutenauer Hopf algebra $\FQSym$ of Free quasi-symmetric
functions~\cite{MR95, DHT02} has bases indexed by permutations. This
Hopf algebra admits several Hopf subalgebras: The Hopf algebra of Free
symmetric functions $\FSym$ ~\cite{PR95, DHT02}, whose bases are indexed
by standard Young tableaux, the Hopf algebra $\Bell$~\cite{R07} whose bases
are indexed by set partitions, the Loday-Ronco Hopf algebra $\PBT$~\cite{LR98, HNT05-2}
whose bases are indexed by planar binary trees and the Hopf algebra $\Sym$ of
non-commutative symmetric functions~\cite{GKDLLRT94} whose bases are indexed
by integer compositions. An unifying approach to construct all these structures
relies on a definition of a congruence on words leading to the definition
of monoids on combinatorial objects. Indeed, $\FSym$ is directly obtained
from the plactic monoid~\cite{LS81}, $\Bell$ from the Bell monoid~\cite{R07},
$\PBT$ from the sylvester monoid~\cite{HNT02-2, HNT05-2}, and $\Sym$ from
the hypoplactic monoid~\cite{N98}. The richness of these constructions relies
on the fact that, in addition to construct Hopf algebras, the definition
of such monoids often brings partial orders, combinatorial algorithms and
Robinson-Schensted-like algorithms, of independent interest.

In this paper, we propose to enrich this collection of Hopf algebras by
providing a construction of a Hopf algebra whose bases are indexed by
objects belonging to the Baxter combinatorial family. This combinatorial
family admits various representations as Baxter permutations~\cite{Bax64},
pairs of twin binary trees~\cite{DG94}, quadrangulations~\cite{ABP04}, plane
bipolar orientations~\cite{BBF10}, \emph{etc}. In~\cite{R05}, Reading defines
first a Hopf algebra on Baxter permutations in the context of lattice congruences;
Moreover, very recently, Law and Reading~\cite{LR10} have studied and detailed
their construction of this Hopf algebra. However, even if both points of
view lead to the same general theory, their paths are different and provide
different ways of understanding this Hopf algebra, one centered, as in Law and
Reading's work, on lattice theory, the other, as in our work, centered on
combinatorics on words. Moreover, a large part of the results of each paper
does not appear in the other.

We begin by recalling in Section~\ref{sec:Prelim} the preliminary notions
used thereafter. In Section~\ref{sec:MonoideBaxter}, we define the Baxter
congruence. This congruence allows to define a quotient of the free monoid,
the Baxter monoid, which has a number of properties required for the Hopf
algebraic construction which follows. We show that the Baxter monoid is
intimately linked to the sylvester monoid. Next, in Section~\ref{sec:RobinsonSchensted},
we develop a Robinson-Schensted-like insertion algorithm that allows to
decide if two words are equivalent according to the Baxter congruence.
Given a word, this algorithm computes a pair of twin binary trees. Section~\ref{sec:TreillisBaxter}
is devoted to the study of some properties of the equivalence classes of
permutations under the Baxter congruence. This leads to the definition of
a lattice structure on pairs of twin binary trees. Finally, in Section~\ref{sec:AlgebreHopfBaxter},
we define the Hopf algebra $\Baxter$ and study it. Using the order structure
on pairs of twin binary trees, we provide multiplicative bases and show
that $\Baxter$ is free as an algebra. Using the results of Foissy on
bidendriform bialgebras~\cite{F05}, we show that $\Baxter$ is also self-dual
and that the Lie algebra of its primitive elements is free.

\acknowledgements
The author would like to thank Florent Hivert and Jean-Christophe Novelli
for their advice and help during all stages of the preparation of this paper.
The computations of this work have been done with the open-source mathematical
software Sage~\cite{SAGE}.

\section{Preliminaries} \label{sec:Prelim}

\subsection{Words}
In the sequel, $A := \{a_1 < a_2 < \ldots\}$ is a totally ordered infinite
alphabet and $A^*$ is the free monoid spanned by $A$. Let $u \in A^*$.
For $S \subseteq A$, we denote by $u_{|S}$ the \emph{restriction} of $u$
on the alphabet $S$, that is the longest subword of $u$ made of letters of $S$.
The \emph{evaluation} $\Eval(u)$ of the word $u$ is the non-negative integer
vector such that its $i$-th entry is the number of occurrences of the letter
$a_i$ in $u$. Let $\max(u)$ be the maximal letter of $u$. The \emph{Schützenberger transformation}
$\#$ is defined by $u^\# := \max(u)\! +\! 1\! -\! u_{|u|} \ldots \max(u)\! +\! 1\! -\! u_1$;
For example, $(a_5 a_3 a_1 a_1 a_5 a_2)^\# = a_4 a_1 a_5 a_5 a_3 a_1$. Note
that it is an involution if $u$ has an occurrence of $a_1$. Let $v \in A^*$
and $\LettreA, \LettreB \in A$. The \emph{shuffle product} $\shuffle$ is
defined on $\EnsRel \langle A \rangle$ recursively by $u \shuffle \epsilon := \epsilon \shuffle u := u$ and
$\LettreA u \shuffle \LettreB v := \LettreA (u \shuffle \LettreB v) + \LettreB (\LettreA u \shuffle v)$.

\subsection{Permutations}
Denote by $\EnsPermu_n$ the set of permutations of size $n$ and
$\EnsPermu := \cup_{n \geq 0} \EnsPermu_n$. We shall call $(i, j)$ a
\emph{co-inversion} of $\sigma \in \EnsPermu$ if $i < j$ and $\sigma^{-1}_i > \sigma^{-1}_j$.
Let us recall that the \emph{(right) permutohedron order} is the partial
order $\OrdPermu$ defined on $\EnsPermu_n$ where $\sigma$ is covered by
$\nu$ if $\sigma = u \LettreA \LettreB v$ and $\nu = u \LettreB \LettreA v$
where $\LettreA < \LettreB$. Let $\sigma, \nu \in \EnsPermu$. The permutation
$\sigma \Over \nu$ is obtained by concatenating $\sigma$ and the letters of $\nu$
incremented by $|\sigma|$; In the same way, the permutation $\sigma \Under \nu$ is
obtained by concatenating the letters of $\nu$ incremented by $|\sigma|$ and
$\sigma$; For example, ${\bf 312} \Over 2314 = {\bf 312} 5647$ and
${\bf 312} \Under 2314 = 5647 {\bf 312}$. The permutation $\sigma$ is
\emph{connected} if $\sigma = \nu \Over \pi$ implies $\nu = \sigma$ or
$\pi = \sigma$. The \emph{shifted shuffle product} $\cshuffle$ of two permutations is
defined by $\sigma \cshuffle \nu := \sigma \shuffle (\nu_1\! +\! |\sigma| \ldots \nu_{|\nu|}\! +\! |\sigma|)$;
For example, ${\bf 12} \cshuffle 21 = {\bf 12} \shuffle 43 = {\bf 12} 43 + {\bf 1} 4 {\bf 2} 3 +
{\bf 1} 43 {\bf 2} + 4 {\bf 12} 3 + 4 {\bf 1} 3 {\bf 2} + 43 {\bf 12}$.
The \emph{standardized word} $\Std(u)$ of $u \in A^*$ is the unique permutation
$\sigma$ satisfying $\sigma_i < \sigma_j$ iff $u_i \leq u_j$ for all
$1 \leq i < j \leq |u|$; For example, $\Std(a_3 a_1 a_4 a_2 a_5 a_7 a_4 a_2 a_3) = 416289735$.

\subsection{Binary trees}
Denote by $\EnsAB_n$ the set of binary trees with $n$ internal nodes and
$\EnsAB := \cup_{n \geq 0} \EnsAB_n$. We use in the sequel the standard
terminology (\emph{i.e.}, \emph{child}, \emph{ancestor}, \ldots)
about binary trees~\cite{AU93}. The only element of $\EnsAB_0$
is the \emph{leaf} or \emph{empty tree}, denoted by $\ArbreVide$. Let us
recall that the \emph{Tamari order}~\cite{K06} is the partial order $\OrdTam$
defined on $\EnsAB_n$ where $T_0 \in \EnsAB_n$ is covered by $T_1 \in \EnsAB_n$
if it is possible to transform $T_0$ into $T_1$ by performing a right rotation
(see Figure~\ref{fig:Rotation}).
\begin{figure}[ht]
    \centering
    \scalebox{.25}{
        \begin{tikzpicture}
            \node[Noeud, EtiqClair, minimum size = 12 mm] (racine) at (0, 0) {};
            \node (g) at (-3, -1) {};
            \node (d) at (3, -1) {};
            \node[Noeud, EtiqClair, minimum size = 12 mm] (r) at (0, -2) {$y$};
            \node[Noeud, EtiqClair, minimum size = 12 mm] (q) at (-2, -4) {$x$};
            \node[SArbre, minimum size = 16 mm] (A) at (-4, -6) {\Huge $A$};
            \node[SArbre, minimum size = 16 mm] (B) at (0, -6) {\Huge $B$};
            \node[SArbre, minimum size = 16 mm] (C) at (2, -4) {\Huge $C$};
            \draw[Arete] (racine) -- (g);
            \draw[Arete] (racine) -- (d);
            \draw[Arete, decorate, decoration = zigzag] (racine) -- (r);
            \draw[Arete] (r) -- (q);
            \draw[Arete] (r) -- (C);
            \draw[Arete] (q) -- (A);
            \draw[Arete] (q) -- (B);
            \node at (-4, -3) {\scalebox{3}{$T_0 = $}};
            \draw[line width=3pt, ->] (4, -3) -- (6, -3);
            \node[Noeud, EtiqClair, minimum size = 12 mm] (racine') at (10, 0) {};
            \node (g') at (7, -1) {};
            \node (d') at (13, -1) {};
            \node[Noeud, EtiqClair, minimum size = 12 mm] (r') at (12, -4) {$y$};
            \node[Noeud, EtiqClair, minimum size = 12 mm] (q') at (10, -2) {$x$};
            \node[SArbre, minimum size = 16 mm] (A') at (8, -4) {\Huge $A$};
            \node[SArbre, minimum size = 16 mm] (B') at (10, -6) {\Huge $B$};
            \node[SArbre, minimum size = 16 mm] (C') at (14, -6) {\Huge $C$};
            \draw[Arete] (racine') -- (g');
            \draw[Arete] (racine') -- (d');
            \draw[Arete, decorate, decoration = zigzag] (racine') -- (q');
            \draw[Arete] (q') -- (r');
            \draw[Arete] (q') -- (A');
            \draw[Arete] (r') -- (B');
            \draw[Arete] (r') -- (C');
            \node at (14, -3) {\scalebox{3}{$= T_1$}};
        \end{tikzpicture}
    }
    \caption{The right rotation of root $y$.}
    \label{fig:Rotation}
\end{figure}
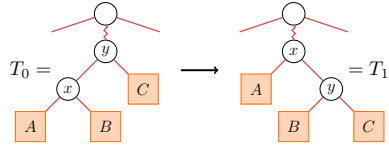

Let $T_0, T_1 \in \EnsAB$. The binary tree $T_0 \Over T_1$ is obtained by
grafting $T_0$ from its root on the leftmost leaf of $T_1$; In the same way,
the binary tree $T_0 \Under T_1$ is obtained by grafting $T_1$ from its root
on the rightmost leaf of $T_0$. The \emph{canopy} (see~\cite{LR98} and~\cite{V04})
$\Canop(T)$ of $T \in \EnsAB$ is the word on the alphabet $\{0, 1\}$ obtained
by browsing the leaves of $T$ from left to right except the first and the
last one, writing $0$ if the considered leaf is oriented to the right, $1$
otherwise (see Figure~\ref{fig:ExempleCanopee}). Note that the orientation
of the leaves in a binary tree is determined only by its nodes so that we
can omit to draw the leaves in our next graphical representations.
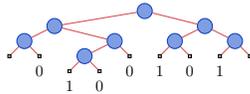
\begin{figure}[ht]
    \centering
    \scalebox{.20}{
        \begin{tikzpicture}
            \node[Feuille](0)at(0,-3){};
            \node[Noeud](1)at(1,-2){};
            \node[Feuille](2)at(2,-3){};
            \node[] (2') [below of = 2] {\scalebox{3}{$0$}};
            \draw[Arete](1)--(0);
            \draw[Arete](1)--(2);
            \node[Noeud](3)at(3,-1){};
            \node[Feuille](4)at(4,-4){};
            \node[] (4') [below of = 4] {\scalebox{3}{$1$}};
            \node[Noeud](5)at(5,-3){};
            \node[Feuille](6)at(6,-4){};
            \node[] (6') [below of = 6] {\scalebox{3}{$0$}};
            \draw[Arete](5)--(4);
            \draw[Arete](5)--(6);
            \node[Noeud](7)at(7,-2){};
            \node[Feuille](8)at(8,-3){};
            \node[] (8') [below of = 8] {\scalebox{3}{$0$}};
            \draw[Arete](7)--(5);
            \draw[Arete](7)--(8);
            \draw[Arete](3)--(1);
            \draw[Arete](3)--(7);
            \node[Noeud](9)at(9,0){};
            \node[Feuille](10)at(10,-3){};
            \node[] (10') [below of = 10] {\scalebox{3}{$1$}};
            \node[Noeud](11)at(11,-2){};
            \node[Feuille](12)at(12,-3){};
            \node[] (12') [below of = 12] {\scalebox{3}{$0$}};
            \draw[Arete](11)--(10);
            \draw[Arete](11)--(12);
            \node[Noeud](13)at(13,-1){};
            \node[Feuille](14)at(14,-3){};
            \node[] (14') [below of = 14] {\scalebox{3}{$1$}};
            \node[Noeud](15)at(15,-2){};
            \node[Feuille](16)at(16,-3){};
            \draw[Arete](15)--(14);
            \draw[Arete](15)--(16);
            \draw[Arete](13)--(11);
            \draw[Arete](13)--(15);
            \draw[Arete](9)--(3);
            \draw[Arete](9)--(13);
        \end{tikzpicture}
    }
    \caption{The canopy of this binary tree is $0100101$.}
    \label{fig:ExempleCanopee}
\end{figure}

An $A$-labeled binary tree $T$ is a \emph{left} (resp. \emph{right})
\emph{binary search tree} if for any node $x$ labeled by $\LettreB$, each
label $\LettreA$ of a node in the left subtree of $x$ and each label $\LettreC$
of a node in the right subtree of $x$, the inequality $\LettreA < \LettreB \leq \LettreC$
(resp. $\LettreA \leq \LettreB < \LettreC$) holds. A binary tree $T \in \EnsAB_n$
is a \emph{decreasing binary tree} if it is bijectively labeled on
$\{1, \ldots, n\}$ and, for all node $y$ of $T$, if $x$ is a child of $y$,
then the label of $x$ is smaller than the label of $y$. The \emph{shape}
of a labeled binary tree is the unlabeled binary tree obtained by forgetting
its labels.

\subsection{Baxter permutations and pairs of twin binary trees}
A permutation $\sigma$ is a \emph{Baxter permutation} if for any subword
$u = u_1 u_2 u_3 u_4$ of $\sigma$ such that the letters $u_2$ and $u_3$
are adjacent in $\sigma$, $\Std(u) \notin \{2413, 3142\}$. In other words,
$\sigma$ is a Baxter permutation if it avoids the \emph{generalized permutation patterns}
$2-41-3$ and $3-14-2$ (see~\cite{BS00} for an introduction on generalized permutation
patterns). For example, ${\bf 4}21{\bf 73}8{\bf 5}6$ is not a Baxter permutation;
On the other hand $436975128$ is a Baxter permutation. Let us denote by $\EnsPermuBX_n$
the set of Baxter permutations of size $n$ and $\EnsPermuBX := \cup_{n \geq 0} \EnsPermuBX_n$.

A \emph{pair of twin binary trees} $(T_L, T_R)$ is made of two binary trees
$T_L, T_R \in \EnsAB_n$ such that the canopies of $T_L$ and $T_R$ are complementary,
that is $\Canop(T_L)_i \ne \Canop(T_R)_i$ for all $1 \leq i \leq n - 1$.
Denote by $\EnsABJ_n$ the set of pairs of twin binary trees where each binary
tree has $n$ nodes and $\EnsABJ := \cup_{n \geq 0} \EnsABJ_n$. In~\cite{DG94},
Dulucq and Guibert have highlighted a bijection between Baxter permutations
and pairs of twin binary trees. In the sequel, we shall make use of a
very similar bijection.

\section{The Baxter monoid} \label{sec:MonoideBaxter}

\subsection{Definition and first properties}
Recall that an equivalence relation $\equiv$ defined on $A^*$ is a \emph{congruence}
if for all $u, u', v, v' \in A^*$, $u \equiv u'$ and $v \equiv v'$  imply
$u . v \equiv u' . v'$.

\begin{Definition} \label{def:MonoideBaxter}
    The \emph{Baxter monoid} is the quotient of the free monoid $A^*$ by
    the congruence $\EquivBX$ that is the transitive closure of the \emph{adjacency relations}
    $\AdjBXA$ and $\AdjBXB$ defined for $u, v\in A^*$ and $\LettreA, \LettreB, \LettreC, \LettreD \in A$~by:
    \begin{align}
        \LettreC u \LettreA \LettreD v \LettreB & \AdjBXA \LettreC u \LettreD \LettreA v \LettreB
                \hspace{4em} \text{where \hspace{1em} $\LettreA \leq \LettreB < \LettreC \leq \LettreD$,} \label{eq:EquivBXAdj1} \\
        \LettreB u \LettreD \LettreA v \LettreC & \AdjBXB \LettreB u \LettreA \LettreD v \LettreC
                \hspace{4em} \text{where \hspace{1em} $\LettreA < \LettreB \leq \LettreC < \LettreD$.} \label{eq:EquivBXAdj2}
    \end{align}
\end{Definition}

For $u \in A^*$, denote by $\widehat{u}$ the $\EquivBX$-equivalence class
of $u$; For example, the $\EquivBX$-equivalence class of $5273641$ is
$\{5237641, 5273641, 5276341, 5723641, 5726341, 5762341\}$.

An equivalence relation $\equiv$ defined on $A^*$ is \emph{compatible with
the restriction of alphabet intervals} if for all interval $I$ of $A$ and
for all $u, v \in A^*$, $u \equiv v$ implies $u_{|I} \equiv v_{|I}$.

\begin{Proposition} \label{prop:CompRestrSegmAlph}
    The Baxter monoid is compatible with the restriction of alphabet intervals.
\end{Proposition}
\begin{proof}
    We only have to check the property on adjacency relations.
\end{proof}

An equivalence relation $\equiv$ defined on $A^*$ is \emph{compatible
with the destandardization process} if for all $u, v \in A^*$,
$u \equiv v$ iff $\Std(u) \equiv \Std(v)$ and $\Eval(u) = \Eval(v)$.

\begin{Proposition} \label{prop:CompDestd}
    The Baxter monoid is compatible with the destandardization process.
\end{Proposition}

An equivalence relation $\equiv$ defined on $A^*$ is \emph{compatible with
the Schützenberger involution} if for all $u, v \in A^*$, $u \equiv v$
implies $u^\# \equiv v^\#$.

\begin{Proposition} \label{prop:CompSchutz}
    The Baxter monoid is compatible with the Schützenberger involution.
\end{Proposition}

\subsection{Connection with the sylvester monoid}
The \emph{sylvester monoid}~\cite{HNT02-2, HNT05-2} is the quotient of the
free monoid $A^*$ by the congruence $\EquivS$ that is the transitive
closure of the adjacency relation $\AdjS$ defined for $u \in A^*$ and
$\LettreA, \LettreB, \LettreC \in A$ by:
\begin{equation}
    \LettreA \LettreC u \LettreB \AdjS \LettreC \LettreA u \LettreB
            \hspace{4em} \text{where \hspace{1em} $\LettreA \leq \LettreB < \LettreC$.}
\end{equation}
In the same way, let us define the \emph{$\#$-sylvester monoid} by the congruence
$\EquivCoS$ that is the transitive closure of the adjacency relation $\AdjCoS$
defined for $u\in A^*$ and $\LettreA, \LettreB, \LettreC \in A$ by:
\begin{equation}
    \LettreB u \LettreA \LettreC \AdjCoS \LettreB u \LettreC \LettreA
            \hspace{4em} \text{where \hspace{1em} $\LettreA < \LettreB \leq \LettreC$.}
\end{equation}
Note that this adjacency relation is defined by taking the images by the
Schützenberger involution of the sylvester adjacency relation. Indeed,
for all $u, v \in A^*$, $u \EquivCoS v$ iff $u^\# \EquivS v^\#$. The Baxter
monoid and the sylvester monoid are related in the following way:
\begin{Proposition} \label{prop:LienSylv}
    Let $u, v \in A^*$. Then, $u \EquivBX v$ iff $u \EquivS v$ and $u \EquivCoS v$.
\end{Proposition}

Proposition~\ref{prop:LienSylv} shows that the $\EquivBX$-equivalence classes
are the intersection of $\EquivS$-equivalence classes and $\EquivCoS$-equivalence
classes.

\section{A Robinson-Schensted-like algorithm} \label{sec:RobinsonSchensted}

We shall describe here an insertion algorithm $u \mapsto \left(\PSymb(u), \QSymb(u)\right)$,
such that, given a word $u \in A^*$, it computes its $\PSymb$-symbol, that
is a pair of $A$-labeled twin binary trees $(T_L, T_R)$ where $T_L$ (resp. $T_R$)
is a left (resp. right) binary search tree, and its $\QSymb$-symbol, a
decreasing binary tree.

\subsection{Definition of the insertion algorithm}
Let $T$ be an $A$-labeled right binary search tree and $\LettreB$  a letter
of $A$. The \emph{lower restricted binary tree} of $T$ compared to $\LettreB$,
namely $T_{\leq \LettreB}$, is the right binary search tree uniquely made
of the nodes $x$ of $T$ labeled by a letter $\LettreA$ satisfying $\LettreA \leq \LettreB$
and such that for all nodes $x$ and $y$ of $T_{\leq \LettreB}$, if $x$ is
ancestor of $y$ in $T_{\leq \LettreB}$, then $x$ is ancestor of $y$ in $T$.
In the same way, we define the \emph{higher restricted binary tree} of $T$
compared to $\LettreB$, namely $T_{> \LettreB}$ (see Figure~\ref{fig:ExempleABRestreints}).
\begin{figure}[ht]
    \centering
    \scalebox{.35}{
        \begin{tikzpicture}
            \node[Noeud,EtiqFonce](0)at(0,-2){$1$};
            \node[Noeud,EtiqFonce](1)at(1,-1){$1$};
            \draw[Arete](1)--(0);
            \node[Noeud,EtiqFonce](2)at(2,-3){$2$};
            \node[Noeud,EtiqClair](3)at(3,-4){$3$};
            \draw[Arete](2)--(3);
            \node[Noeud,EtiqClair](4)at(4,-2){$3$};
            \draw[Arete](4)--(2);
            \draw[Arete](1)--(4);
            \node[Noeud,EtiqClair](5)at(5,0){$4$};
            \draw[Arete](5)--(1);
            \node[Noeud,EtiqClair](6)at(6,-1){$5$};
            \draw[Arete](5)--(6);
        \end{tikzpicture}
        \hspace{3em}
        \raisebox{3em}{\begin{tikzpicture}
            \node[Noeud,EtiqFonce](0)at(0,-1){$1$};
            \node[Noeud,EtiqFonce](1)at(1,0){$1$};
            \draw[Arete](1)--(0);
            \node[Noeud,EtiqFonce](2)at(2,-1){$2$};
            \draw[Arete](1)--(2);
        \end{tikzpicture}}
        \hspace{3em}
        \raisebox{2em}{\begin{tikzpicture}
            \node[Noeud,EtiqClair](0)at(0,-2){$3$};
            \node[Noeud,EtiqClair](1)at(1,-1){$3$};
            \draw[Arete](1)--(0);
            \node[Noeud,EtiqClair](2)at(2,0){$4$};
            \draw[Arete](2)--(1);
            \node[Noeud,EtiqClair](3)at(3,-1){$5$};
            \draw[Arete](2)--(3);
        \end{tikzpicture}}
    }
    \caption{A right binary search tree $T$, $T_{\leq 2}$ and $T_{>2}$.}
    \label{fig:ExempleABRestreints}
\end{figure}

Let $T$ be an $A$-labeled right binary search tree and $\LettreA$ a letter
of $A$. The \emph{root insertion} of $\LettreA$ into $T$ consists in modifying $T$
so that the root of $T$ is a new node labeled by $\LettreA$, its left subtree is
$T_{\leq \LettreA}$ and its right subtree is $T_{> \LettreA}$.

Let $T$ be an $A$-labeled left (resp. right) binary search tree and $\LettreA$ a letter
of $A$. The \emph{leaf insertion} of $\LettreA$ into $T$ is recursively defined
by: If $T = \ArbreVide$, the result is the one-node binary tree labeled by $\LettreA$;
Else, if the label $\LettreB$ of the root of $T$ satisfies $\LettreA < \LettreB$
(resp. $\LettreA \leq \LettreB$), make a leaf insertion of $\LettreA$ into
the left subtree of $T$, else, make a leaf insertion of $\LettreA$ into
the right subtree of $T$.

Given a pair of $A$-labeled twin binary trees $(T_L, T_R)$ where $T_L$
(resp. $T_R$) is a left (resp. right) binary search tree, the \emph{insertion}
of the letter $\LettreA$ of $A$ into $(T_L, T_R)$ consists in making a leaf insertion of
$\LettreA$ into $T_L$ and a root insertion of $\LettreA$ into $T_R$.

The $\PSymb$-symbol $(T_L, T_R)$ of a word $u \in A^*$ is computed by
iteratively inserting the letters of $u$, from left to right, into the
pair of twin binary trees $(\ArbreVide, \ArbreVide)$. The $\QSymb$-symbol
of $u$ is the decreasing binary tree labeled on $\{1, \ldots, |u|\}$,
built by recording the dates of creation of each node of $T_R$
(see Figure~\ref{fig:ExemplePQSymbole}).

\begin{figure}[ht]
    \centering
    \begin{tabular}{cc}
    $\ArbreVide \ArbreVide$
    \hspace{1em}$\xrightarrow{5}$\hspace{1em}
    \scalebox{.34}{
        \begin{tikzpicture}
            \node[Noeud,EtiqFonce](0)at(0,0){$5$};
        \end{tikzpicture}
        \begin{tikzpicture}
            \node[Noeud,EtiqFonce](0)at(0,0){$5$};
        \end{tikzpicture}
    }
    \hspace{1em}$\xrightarrow{4}$\hspace{1em}
    \scalebox{.33}{
        \raisebox{-2em}{
        \begin{tikzpicture}
            \node[Noeud,EtiqFonce](0)at(0,-1){$4$};
            \node[Noeud,EtiqClair](1)at(1,0){$5$};
            \draw[Arete](1)--(0);
        \end{tikzpicture}
        \begin{tikzpicture}
            \node[Noeud,EtiqFonce](0)at(0,0){$4$};
            \node[Noeud,EtiqClair](1)at(1,-1){$5$};
            \draw[Arete](0)--(1);
        \end{tikzpicture}}
    }
    \hspace{1em}$\xrightarrow{2}$\hspace{1em}
    \scalebox{.33}{
        \raisebox{-3.3em}{
        \begin{tikzpicture}
            \node[Noeud,EtiqFonce](0)at(0,-2){$2$};
            \node[Noeud,EtiqClair](1)at(1,-1){$4$};
            \draw[Arete](1)--(0);
            \node[Noeud,EtiqClair](2)at(2,0){$5$};
            \draw[Arete](2)--(1);
        \end{tikzpicture}
        \begin{tikzpicture}
            \node[Noeud,EtiqFonce](0)at(0,0){$2$};
            \node[Noeud,EtiqClair](1)at(1,-1){$4$};
            \node[Noeud,EtiqClair](2)at(2,-2){$5$};
            \draw[Arete](1)--(2);
            \draw[Arete](0)--(1);
        \end{tikzpicture}}
    }
    \hspace{1em}$\xrightarrow{5}$ \\[2em]
    \scalebox{.33}{
        \raisebox{-3.3em}{
        \begin{tikzpicture}
            \node[Noeud,EtiqClair](0)at(0,-2){$2$};
            \node[Noeud,EtiqClair](1)at(1,-1){$4$};
            \draw[Arete](1)--(0);
            \node[Noeud,EtiqClair](2)at(2,0){$5$};
            \draw[Arete](2)--(1);
            \node[Noeud,EtiqFonce](3)at(3,-1){$5$};
            \draw[Arete](2)--(3);
        \end{tikzpicture}
        \begin{tikzpicture}
            \node[Noeud,EtiqClair](0)at(0,-1){$2$};
            \node[Noeud,EtiqClair](1)at(1,-2){$4$};
            \node[Noeud,EtiqClair](2)at(2,-3){$5$};
            \draw[Arete](1)--(2);
            \draw[Arete](0)--(1);
            \node[Noeud,EtiqFonce](3)at(3,0){$5$};
            \draw[Arete](3)--(0);
        \end{tikzpicture}}
    }
    \hspace{1em}$\xrightarrow{4}$\hspace{1em}
    \scalebox{.33}{
        \raisebox{-3.3em}{
        \begin{tikzpicture}
            \node[Noeud,EtiqClair](0)at(0,-2){$2$};
            \node[Noeud,EtiqClair](1)at(1,-1){$4$};
            \draw[Arete](1)--(0);
            \node[Noeud,EtiqFonce](2)at(2,-2){$4$};
            \draw[Arete](1)--(2);
            \node[Noeud,EtiqClair](3)at(3,0){$5$};
            \draw[Arete](3)--(1);
            \node[Noeud,EtiqClair](4)at(4,-1){$5$};
            \draw[Arete](3)--(4);
        \end{tikzpicture}
        \begin{tikzpicture}
            \node[Noeud,EtiqClair](0)at(0,-1){$2$};
            \node[Noeud,EtiqClair](1)at(1,-2){$4$};
            \draw[Arete](0)--(1);
            \node[Noeud,EtiqFonce](2)at(2,0){$4$};
            \draw[Arete](2)--(0);
            \node[Noeud,EtiqClair](3)at(3,-2){$5$};
            \node[Noeud,EtiqClair](4)at(4,-1){$5$};
            \draw[Arete](4)--(3);
            \draw[Arete](2)--(4);
        \end{tikzpicture}}
    }
    \hspace{1em}$\xrightarrow{2}$ \\[2em]
    \scalebox{.33}{
        \raisebox{-5.3em}{
        \begin{tikzpicture}
            \node[Noeud,EtiqClair](0)at(0,-2){$2$};
            \node[Noeud,EtiqFonce](1)at(1,-3){$2$};
            \draw[Arete](0)--(1);
            \node[Noeud,EtiqClair](2)at(2,-1){$4$};
            \draw[Arete](2)--(0);
            \node[Noeud,EtiqClair](3)at(3,-2){$4$};
            \draw[Arete](2)--(3);
            \node[Noeud,EtiqClair](4)at(4,0){$5$};
            \draw[Arete](4)--(2);
            \node[Noeud,EtiqClair](5)at(5,-1){$5$};
            \draw[Arete](4)--(5);
        \end{tikzpicture}
        \begin{tikzpicture}
            \node[Noeud,EtiqClair](0)at(0,-1){$2$};
            \node[Noeud,EtiqFonce](1)at(1,0){$2$};
            \draw[Arete](1)--(0);
            \node[Noeud,EtiqClair](2)at(2,-2){$4$};
            \node[Noeud,EtiqClair](3)at(3,-1){$4$};
            \draw[Arete](3)--(2);
            \node[Noeud,EtiqClair](4)at(4,-3){$5$};
            \node[Noeud,EtiqClair](5)at(5,-2){$5$};
            \draw[Arete](5)--(4);
            \draw[Arete](3)--(5);
            \draw[Arete](1)--(3);
        \end{tikzpicture}}
    }
    $\xrightarrow{4}$
    \scalebox{.33}{
        \raisebox{-5.3em}{
        \begin{tikzpicture}
            \node[Noeud,EtiqClair](0)at(0,-2){$2$};
            \node[Noeud,EtiqClair](1)at(1,-3){$2$};
            \draw[Arete](0)--(1);
            \node[Noeud,EtiqClair](2)at(2,-1){$4$};
            \draw[Arete](2)--(0);
            \node[Noeud,EtiqClair](3)at(3,-2){$4$};
            \node[Noeud,EtiqFonce](4)at(4,-3){$4$};
            \draw[Arete](3)--(4);
            \draw[Arete](2)--(3);
            \node[Noeud,EtiqClair](5)at(5,0){$5$};
            \draw[Arete](5)--(2);
            \node[Noeud,EtiqClair](6)at(6,-1){$5$};
            \draw[Arete](5)--(6);
        \end{tikzpicture}
        \begin{tikzpicture}
            \node[Noeud,EtiqClair](0)at(0,-2){$2$};
            \node[Noeud,EtiqClair](1)at(1,-1){$2$};
            \draw[Arete](1)--(0);
            \node[Noeud,EtiqClair](2)at(2,-3){$4$};
            \node[Noeud,EtiqClair](3)at(3,-2){$4$};
            \draw[Arete](3)--(2);
            \draw[Arete](1)--(3);
            \node[Noeud,EtiqFonce](4)at(4,0){$4$};
            \draw[Arete](4)--(1);
            \node[Noeud,EtiqClair](5)at(5,-2){$5$};
            \node[Noeud,EtiqClair](6)at(6,-1){$5$};
            \draw[Arete](6)--(5);
            \draw[Arete](4)--(6);
        \end{tikzpicture}}
    }
    {\footnotesize $ = \PSymb(u)$;}
    &
    \scalebox{.33}{
        \raisebox{-5.3em}{
        \begin{tikzpicture}
            \node[Noeud,EtiqClair](0)at(0,-2){$3$};
            \node[Noeud,EtiqClair](1)at(1,-1){$6$};
            \draw[Arete](1)--(0);
            \node[Noeud,EtiqClair](2)at(2,-3){$2$};
            \node[Noeud,EtiqClair](3)at(3,-2){$5$};
            \draw[Arete](3)--(2);
            \draw[Arete](1)--(3);
            \node[Noeud,EtiqClair](4)at(4,0){$7$};
            \draw[Arete](4)--(1);
            \node[Noeud,EtiqClair](5)at(5,-2){$1$};
            \node[Noeud,EtiqClair](6)at(6,-1){$4$};
            \draw[Arete](6)--(5);
            \draw[Arete](4)--(6);
        \end{tikzpicture}}
    }
    {\footnotesize $ = \QSymb(u)$}
    \end{tabular}
    \caption{Steps of computation of the $\PSymb$-symbol and the $\QSymb$-symbol of $u := 5425424$.}
    \label{fig:ExemplePQSymbole}
\end{figure}

\subsection{Validity of the insertion algorithm}

\begin{Lemme} \label{lem:Insertion}
    Let $u \in A^*$. Let $T$ be the right binary search tree obtained
    by root insertions of the letters of $u$, from left to right.
    Let $T'$ be the right binary search tree obtained by leaf insertions
    of the letters of $u$, from right to left. Then, $T = T'$.
\end{Lemme}

\begin{Lemme} \label{lem:FeuillesInversions}
    Let $\sigma \in \EnsPermu$ and $T \in \EnsAB_{|\sigma|}$ be the binary search tree obtained
    by leaf insertions of the letters of $\sigma$, from left to right.
    Then, for $1 \leq i \leq |\sigma|\!-\!1$, the $i\!+\!1$-st leaf of $T$ is
    right-oriented iff $(i, i\!+\!1)$ is a co-inversion of $\sigma$.
\end{Lemme}

If $(T_L, T_R)$ is a pair of labeled twin binary trees, define its
\emph{shape}, that is the pair of unlabeled twin binary trees $(T'_L, T'_R)$
where $T'_L$ (resp. $T'_R$) is the shape of $T_L$ (resp. $T_R$).

\begin{Proposition} \label{prop:PSymboleCoupleJumeaux}
    For all word $u \in A^*$, the shape of the $\PSymb$-symbol of $u$ is a pair of
    twin binary trees.
\end{Proposition}

\begin{Proposition} \label{prop:PSymboleClasses}
    Let $u, v \in A^*$. Then, $u \EquivBX v$ iff $\PSymb(u) = \PSymb(v)$.
\end{Proposition}

In particular, we have $\PSymb(\sigma) = \PSymb(\nu)$ iff the permutations
$\sigma$ and $\nu$ are $\EquivBX$-equivalent. Moreover, each $\EquivBX$-equivalence
class of permutations can be encoded by a pair of unlabeled twin binary
trees because there is one unique way to bijectively label a binary tree
with $n$ nodes on $\{1, \ldots, n\}$ such that it is a binary search tree.

\begin{Remarque}
    Let $u, v \in A^*$ and $(T_L, T_R) := \PSymb(u)$. We have $u \EquivBX v$
    iff the following two assertions are satisfied:
    \begin{enumerate}[label=(\roman*)]
        \item $v$ is a linear extension of $T_L$ seen as a poset in which
        the smallest element is its root;
        \item $v$ is a linear extension of $T_R$ seen as a poset in which
        minimal elements are the nodes with no descents.
    \end{enumerate}
\end{Remarque}

\section{The Baxter lattice} \label{sec:TreillisBaxter}

\subsection{\texorpdfstring{Some properties of the $\EquivBX$-equivalence classes of permutations}
                           {Some properties of the Baxter equivalence classes of permutations}}

\begin{Theoreme} \label{thm:EquivBXBaxter}
    For all $n\geq 0$, each equivalence class of $\EnsPermu_n /_\EquivBX$
    contains exactly one Baxter permutation.
\end{Theoreme}

\begin{Proposition} \label{prop:EquivBXInter}
    For all $n \geq 0$, each equivalence class of $\EnsPermu_n /_\EquivBX$
    is an interval of the permutohedron.
\end{Proposition}

For all permutation $\sigma$, let us define $\sigma \uparrow$ (resp. $\sigma \downarrow$)
the maximal (resp. minimal) permutation of the $\EquivBX$-equivalence class of $\sigma$
for the permutohedron order.

\begin{Proposition} \label{prop:OrdreBaxterMinMax}
    Let $\sigma, \nu \in \EnsPermu_n$ such that $\sigma \OrdPermu \nu$.
    Then, $\sigma \uparrow \OrdPermu \nu \uparrow$ and $\sigma \downarrow \OrdPermu \nu \downarrow$.
\end{Proposition}

\subsection{A lattice structure on the set of pairs of twin binary trees}

\begin{Definition}
    For all $n \geq 0$, define the order relation $\OrdBX$ on the set
    $\EnsABJ_n$ setting $J_0 \OrdBX J_1$, where $J_0, J_1 \in \EnsABJ_n$,
    if there exists $\sigma_0, \sigma_1 \in \EnsPermu_n$ such that
    $\PSymb(\sigma_0) = J_0$, $\PSymb(\sigma_1) = J_1$ and
    $\sigma_0 \OrdPermu \sigma_1$.
\end{Definition}

Propositions~\ref{prop:EquivBXInter} and~\ref{prop:OrdreBaxterMinMax} ensure
that this order is well-defined, and in particular that the relation $\OrdBX$
is transitive and antisymmetric.

The pair of twin binary trees $(T_L, T_R)$ is covered by $(T'_L, T'_R) \in \EnsABJ$ if
one of the three following conditions is satisfied:
\begin{enumerate}
    \item $T'_R = T_R$ and $T'_L$ is obtained from $T_L$ by performing a left
    rotation into $T_L$ such that $\Canop(T_L) = \Canop(T'_L)$;
    \item $T'_L = T_L$ and $T'_R$ is obtained from $T_R$ by performing a right
    rotation into $T_R$ such that $\Canop(T_R) = \Canop(T'_R)$;
    \item $T'_L$ (resp. $T'_R$) is obtained by performing a left (resp. right)
    rotation into $T_L$ (resp. $T_R$) such that $\Canop(T_L) \ne \Canop(T'_L)$
    (resp. $\Canop(T_R) \ne \Canop(T'_R)$).
\end{enumerate}

Moreover, it is possible to compare two pairs of twin binary trees $J_0 := (T^0_L, T^0_R)$
and $J_1 := (T^1_L, T^1_R)$ very easily by computing the \emph{Tamari vector}
(see~\cite{K06}) of each binary tree. Indeed, we have $J_0 \OrdBX J_1$ iff
the Tamari vector of $T^0_L$ (resp. $T^0_R$) is greater (resp. smaller)
component by component than the Tamari vector of $T^1_L$ (resp. $T^1_R$).

Propositions~\ref{prop:EquivBXInter} and~\ref{prop:OrdreBaxterMinMax} implies
that that $\EquivBX$ is also a lattice congruence~\cite{CS98, R05}. Thus, since
the permutohedron is a lattice,

\begin{Proposition}
    For all $n \geq 0$, the poset $(\EnsABJ_n, \OrdBX)$ is a lattice.
\end{Proposition}

\section{The Baxter Hopf Algebra} \label{sec:AlgebreHopfBaxter}

In the sequel, all the algebraic structures have a field of characteristic
zero $\K$ as ground field.

\subsection{\texorpdfstring{The Hopf algebra $\FQSym$}
                           {The Hopf algebra FQSym}}
Recall that the family $\left\{\F_\sigma\right\}_{\sigma \in \EnsPermu}$
form the \emph{fundamental} basis of $\FQSym$~\cite{DHT02}. Its product
and its coproduct are defined by:
\begin{equation}
    \F_\sigma \Prod \F_\nu := \sum_{\pi \in \sigma \cshuffle \nu} \F_\pi, \hspace{4em}
    \Delta \left(\F_\sigma\right) := \sum_{0 \leq i \leq |\sigma|} \F_{\Std(\sigma_1 \ldots \sigma_i)} \Tenseur \F_{\Std(\sigma_{i + 1} \ldots \sigma_{|\sigma|})}.
\end{equation}

The following theorem due to Hivert and Nzeutchap~\cite{HN07} shows that
an equivalence relation on $A^*$ satisfying some properties can be used
to define Hopf subalgebras of $\FQSym$:
\begin{Theoreme} \label{thm:HivertJanvier}
    Let $\equiv$ be an equivalence relation defined on $A^*$. If $\equiv$ is
    a congruence, compatible with the restriction of alphabet intervals
    and compatible with the destandardization process, then, the family
    $\left\{ \PP_{\widehat{\sigma}} \right\}_{\widehat{\sigma} \in \EnsPermu/_\equiv}$
    defined by:
    \begin{equation}
        \PP_{\widehat{\sigma}} := \sum_{\sigma \in \widehat{\sigma}} \F_\sigma \label{eq:EquivFQSym}
    \end{equation}
    spans a Hopf subalgebra of $\FQSym$.
\end{Theoreme}

\subsection{\texorpdfstring{The Hopf algebra $\Baxter$}
                           {The Hopf algebra Baxter}}
By definition, $\EquivBX$ is a congruence, and, by Proposition~\ref{prop:CompRestrSegmAlph}
and~\ref{prop:CompDestd}, $\EquivBX$ checks the conditions of Theorem~\ref{thm:HivertJanvier}.
Moreover, by Proposition~\ref{prop:PSymboleClasses}, the $\EquivBX$-equivalence
classes of permutations can be encoded by pairs of unlabeled twin binary trees.
Hence, we have the following theorem:
\begin{Theoreme} \label{thm:AlgebreHopfBaxter}
    The family  $\left\{\PP_J\right\}_{J \in \EnsABJ}$ defined by:
    \begin{equation}
        \PP_J := \sum_{\substack{\sigma \in \EnsPermu \\ \PSymb(\sigma) = J}} \F_\sigma
    \end{equation}
    spans a Hopf subalgebra of $\FQSym$, namely the Hopf algebra $\Baxter$.
\end{Theoreme}

The Hilbert series of $\Baxter$ is $B(z) := 1 + z + 2z^2 + 6z^3 + 22z^4 + 92z^5 + 422z^6 + 2074z^7 + 10754z^8 + 58202z^9 + 326240z^{10} + 1882960z^{11} + \ldots$,
the generating series of Baxter permutations (sequence \Sloane{A001181} of \cite{SLOANE}).

One has for example,
\begin{equation}
    \PP_{\scalebox{0.15}{
        \begin{tikzpicture}
            \node[Noeud](0)at(0,0){};
            \node[Noeud](1)at(1,-1){};
            \draw[Arete](0)--(1);
        \end{tikzpicture}
        \begin{tikzpicture}
            \node[Noeud](0)at(0,-1){};
            \node[Noeud](1)at(1,0){};
            \draw[Arete](1)--(0);
        \end{tikzpicture}
    }}
    =
    \F_{12},
    \hspace{1.5em}
    \PP_{\scalebox{0.15}{
        \begin{tikzpicture}
            \node[Noeud](0)at(0,-1){};
            \node[Noeud](1)at(1,0){};
            \draw[Arete](1)--(0);
            \node[Noeud](2)at(2,-2){};
            \node[Noeud](3)at(3,-1){};
            \draw[Arete](3)--(2);
            \draw[Arete](1)--(3);
        \end{tikzpicture}
        \begin{tikzpicture}
            \node[Noeud](0)at(0,-1){};
            \node[Noeud](1)at(1,-2){};
            \draw[Arete](0)--(1);
            \node[Noeud](2)at(2,0){};
            \draw[Arete](2)--(0);
            \node[Noeud](3)at(3,-1){};
            \draw[Arete](2)--(3);
        \end{tikzpicture}
    }}
    =
    \F_{2143} + \F_{2413},
    \hspace{1.5em}
    \PP_{\scalebox{0.15}{
        \begin{tikzpicture}
            \node[Noeud](0)at(0,-3){};
            \node[Noeud](1)at(1,-2){};
            \draw[Arete](1)--(0);
            \node[Noeud](2)at(2,-3){};
            \draw[Arete](1)--(2);
            \node[Noeud](3)at(3,-1){};
            \draw[Arete](3)--(1);
            \node[Noeud](4)at(4,0){};
            \draw[Arete](4)--(3);
            \node[Noeud](5)at(5,-1){};
            \draw[Arete](4)--(5);
        \end{tikzpicture}
        \begin{tikzpicture}
        \node[Noeud](0)at(0,-1){};
        \node[Noeud](1)at(1,-2){};
        \draw[Arete](0)--(1);
        \node[Noeud](2)at(2,0){};
        \draw[Arete](2)--(0);
        \node[Noeud](3)at(3,-2){};
        \node[Noeud](4)at(4,-3){};
        \draw[Arete](3)--(4);
        \node[Noeud](5)at(5,-1){};
        \draw[Arete](5)--(3);
        \draw[Arete](2)--(5);
        \end{tikzpicture}
    }}
    =
    \F_{542163} + \F_{542613} + \F_{546213}.
\end{equation}

By Theorem~\ref{thm:HivertJanvier}, the product of $\Baxter$ is well-defined.
We deduce it from the product of $\FQSym$ and obtain
\begin{equation}
    \PP_{J_0} \Prod \PP_{J_1} = \sum_{\substack{\PSymb(\sigma) = J_0, \hspace{.3em}
                                                \PSymb(\nu) = J_1 \\
                                                \pi \hspace{.3em} \in \hspace{.3em} \sigma \cshuffle \nu \hspace{.3em} \cap \hspace{.3em} \EnsPermuBX}}
                                \PP_{\PSymb(\pi)}.
\end{equation}
For example,
\begin{equation}
\begin{split}
    \PP_{\scalebox{0.15}{
        \begin{tikzpicture}
            \node[Noeud](0)at(0,-1){};
            \node[Noeud](1)at(1,-2){};
            \draw[Arete](0)--(1);
            \node[Noeud](2)at(2,0){};
            \draw[Arete](2)--(0);
        \end{tikzpicture}
        \begin{tikzpicture}
            \node[Noeud](0)at(0,-1){};
            \node[Noeud](1)at(1,0){};
            \draw[Arete](1)--(0);
            \node[Noeud](2)at(2,-1){};
            \draw[Arete](1)--(2);
        \end{tikzpicture}
    }}
    \Prod
    \PP_{\scalebox{0.15}{
        \begin{tikzpicture}
            \node[Noeud](0)at(0,0){};
            \node[Noeud](1)at(1,-1){};
            \draw[Arete](0)--(1);
        \end{tikzpicture}
        \begin{tikzpicture}
            \node[Noeud](0)at(0,-1){};
            \node[Noeud](1)at(1,0){};
            \draw[Arete](1)--(0);
        \end{tikzpicture}
    }}
    & =
    \PP_{\scalebox{0.15}{
        \begin{tikzpicture}
            \node[Noeud](0)at(0,-1){};
            \node[Noeud](1)at(1,-2){};
            \draw[Arete](0)--(1);
            \node[Noeud](2)at(2,0){};
            \draw[Arete](2)--(0);
            \node[Noeud](3)at(3,-1){};
            \node[Noeud](4)at(4,-2){};
            \draw[Arete](3)--(4);
            \draw[Arete](2)--(3);
        \end{tikzpicture}
        \begin{tikzpicture}
            \node[Noeud](0)at(0,-3){};
            \node[Noeud](1)at(1,-2){};
            \draw[Arete](1)--(0);
            \node[Noeud](2)at(2,-3){};
            \draw[Arete](1)--(2);
            \node[Noeud](3)at(3,-1){};
            \draw[Arete](3)--(1);
            \node[Noeud](4)at(4,0){};
            \draw[Arete](4)--(3);
        \end{tikzpicture}
    }}
    +
    \PP_{\scalebox{0.15}{
        \begin{tikzpicture}
            \node[Noeud](0)at(0,-1){};
            \node[Noeud](1)at(1,-2){};
            \draw[Arete](0)--(1);
            \node[Noeud](2)at(2,0){};
            \draw[Arete](2)--(0);
            \node[Noeud](3)at(3,-1){};
            \node[Noeud](4)at(4,-2){};
            \draw[Arete](3)--(4);
            \draw[Arete](2)--(3);
        \end{tikzpicture}
        \begin{tikzpicture}
            \node[Noeud](0)at(0,-2){};
            \node[Noeud](1)at(1,-1){};
            \draw[Arete](1)--(0);
            \node[Noeud](2)at(2,-3){};
            \node[Noeud](3)at(3,-2){};
            \draw[Arete](3)--(2);
            \draw[Arete](1)--(3);
            \node[Noeud](4)at(4,0){};
            \draw[Arete](4)--(1);
        \end{tikzpicture}
    }}
    +
    \PP_{\scalebox{0.15}{
        \begin{tikzpicture}
            \node[Noeud](0)at(0,-1){};
            \node[Noeud](1)at(1,-2){};
            \draw[Arete](0)--(1);
            \node[Noeud](2)at(2,0){};
            \draw[Arete](2)--(0);
            \node[Noeud](3)at(3,-1){};
            \node[Noeud](4)at(4,-2){};
            \draw[Arete](3)--(4);
            \draw[Arete](2)--(3);
        \end{tikzpicture}
        \begin{tikzpicture}
            \node[Noeud](0)at(0,-1){};
            \node[Noeud](1)at(1,0){};
            \draw[Arete](1)--(0);
            \node[Noeud](2)at(2,-3){};
            \node[Noeud](3)at(3,-2){};
            \draw[Arete](3)--(2);
            \node[Noeud](4)at(4,-1){};
            \draw[Arete](4)--(3);
            \draw[Arete](1)--(4);
        \end{tikzpicture}
    }} \\
    & +
    \PP_{\scalebox{0.15}{
        \begin{tikzpicture}
            \node[Noeud](0)at(0,-2){};
            \node[Noeud](1)at(1,-3){};
            \draw[Arete](0)--(1);
            \node[Noeud](2)at(2,-1){};
            \draw[Arete](2)--(0);
            \node[Noeud](3)at(3,0){};
            \draw[Arete](3)--(2);
            \node[Noeud](4)at(4,-1){};
            \draw[Arete](3)--(4);
        \end{tikzpicture}
        \begin{tikzpicture}
            \node[Noeud](0)at(0,-2){};
            \node[Noeud](1)at(1,-1){};
            \draw[Arete](1)--(0);
            \node[Noeud](2)at(2,-2){};
            \node[Noeud](3)at(3,-3){};
            \draw[Arete](2)--(3);
            \draw[Arete](1)--(2);
            \node[Noeud](4)at(4,0){};
            \draw[Arete](4)--(1);
        \end{tikzpicture}
    }}
    +
    \PP_{\scalebox{0.15}{
        \begin{tikzpicture}
            \node[Noeud](0)at(0,-2){};
            \node[Noeud](1)at(1,-3){};
            \draw[Arete](0)--(1);
            \node[Noeud](2)at(2,-1){};
            \draw[Arete](2)--(0);
            \node[Noeud](3)at(3,0){};
            \draw[Arete](3)--(2);
            \node[Noeud](4)at(4,-1){};
            \draw[Arete](3)--(4);
        \end{tikzpicture}
        \begin{tikzpicture}
            \node[Noeud](0)at(0,-1){};
            \node[Noeud](1)at(1,0){};
            \draw[Arete](1)--(0);
            \node[Noeud](2)at(2,-2){};
            \node[Noeud](3)at(3,-3){};
            \draw[Arete](2)--(3);
            \node[Noeud](4)at(4,-1){};
            \draw[Arete](4)--(2);
            \draw[Arete](1)--(4);
        \end{tikzpicture}
    }}
    +
    \PP_{\scalebox{0.15}{
        \begin{tikzpicture}
            \node[Noeud](0)at(0,-2){};
            \node[Noeud](1)at(1,-3){};
            \draw[Arete](0)--(1);
            \node[Noeud](2)at(2,-1){};
            \draw[Arete](2)--(0);
            \node[Noeud](3)at(3,0){};
            \draw[Arete](3)--(2);
            \node[Noeud](4)at(4,-1){};
            \draw[Arete](3)--(4);
        \end{tikzpicture}
        \begin{tikzpicture}
            \node[Noeud](0)at(0,-1){};
            \node[Noeud](1)at(1,0){};
            \draw[Arete](1)--(0);
            \node[Noeud](2)at(2,-1){};
            \node[Noeud](3)at(3,-3){};
            \node[Noeud](4)at(4,-2){};
            \draw[Arete](4)--(3);
            \draw[Arete](2)--(4);
            \draw[Arete](1)--(2);
        \end{tikzpicture}
    }}.
\end{split}
\end{equation}

In the same way, we deduce the coproduct of $\Baxter$ from the coproduct
of $\FQSym$ and obtain
\begin{equation}
    \Delta (\PP_J) = \sum_{\substack{\PSymb(\pi) = J \\ \pi = u.v \\
                                     \sigma := \Std(u), \hspace{.3em} \nu := \Std(v) \in \EnsPermuBX}}
                     \PP_{\PSymb(\sigma)} \Tenseur \PP_{\PSymb(\nu)}.
\end{equation}
For example,
\begin{equation}
\begin{split}
    \Delta
    \PP_{\scalebox{0.15}{
        \begin{tikzpicture}
            \node[Noeud](0)at(0,-1){};
            \node[Noeud](1)at(1,0){};
            \draw[Arete](1)--(0);
            \node[Noeud](2)at(2,-2){};
            \node[Noeud](3)at(3,-1){};
            \draw[Arete](3)--(2);
            \draw[Arete](1)--(3);
        \end{tikzpicture}
        \begin{tikzpicture}
            \node[Noeud](0)at(0,-1){};
            \node[Noeud](1)at(1,-2){};
            \draw[Arete](0)--(1);
            \node[Noeud](2)at(2,0){};
            \draw[Arete](2)--(0);
            \node[Noeud](3)at(3,-1){};
            \draw[Arete](2)--(3);
        \end{tikzpicture}
    }}
    & =
    1
    \Tenseur
    \PP_{\scalebox{0.15}{
        \begin{tikzpicture}
            \node[Noeud](0)at(0,-1){};
            \node[Noeud](1)at(1,0){};
            \draw[Arete](1)--(0);
            \node[Noeud](2)at(2,-2){};
            \node[Noeud](3)at(3,-1){};
            \draw[Arete](3)--(2);
            \draw[Arete](1)--(3);
        \end{tikzpicture}
        \begin{tikzpicture}
            \node[Noeud](0)at(0,-1){};
            \node[Noeud](1)at(1,-2){};
            \draw[Arete](0)--(1);
            \node[Noeud](2)at(2,0){};
            \draw[Arete](2)--(0);
            \node[Noeud](3)at(3,-1){};
            \draw[Arete](2)--(3);
        \end{tikzpicture}
    }}
    +
    \PP_{\scalebox{0.15}{
        \begin{tikzpicture}
            \node[Noeud](0)at(0,0){};
        \end{tikzpicture}
        \begin{tikzpicture}
            \node[Noeud](0)at(0,0){};
        \end{tikzpicture}
    }}
    \Tenseur
    \PP_{\scalebox{0.15}{
        \begin{tikzpicture}
            \node[Noeud](0)at(0,0){};
            \node[Noeud](1)at(1,-2){};
            \node[Noeud](2)at(2,-1){};
            \draw[Arete](2)--(1);
            \draw[Arete](0)--(2);
        \end{tikzpicture}
        \begin{tikzpicture}
            \node[Noeud](0)at(0,-1){};
            \node[Noeud](1)at(1,0){};
            \draw[Arete](1)--(0);
            \node[Noeud](2)at(2,-1){};
            \draw[Arete](1)--(2);
        \end{tikzpicture}
    }}
    +
    \PP_{\scalebox{0.15}{
        \begin{tikzpicture}
            \node[Noeud](0)at(0,0){};
        \end{tikzpicture}
        \begin{tikzpicture}
            \node[Noeud](0)at(0,0){};
        \end{tikzpicture}
    }}
    \Tenseur
    \PP_{\scalebox{0.15}{
        \begin{tikzpicture}
            \node[Noeud](0)at(0,-1){};
            \node[Noeud](1)at(1,-2){};
            \draw[Arete](0)--(1);
            \node[Noeud](2)at(2,0){};
            \draw[Arete](2)--(0);
        \end{tikzpicture}
        \begin{tikzpicture}
            \node[Noeud](0)at(0,-1){};
            \node[Noeud](1)at(1,0){};
            \draw[Arete](1)--(0);
            \node[Noeud](2)at(2,-1){};
            \draw[Arete](1)--(2);
        \end{tikzpicture}
    }}
    +
    \PP_{\scalebox{0.15}{
        \begin{tikzpicture}
            \node[Noeud](0)at(0,0){};
            \node[Noeud](1)at(1,-1){};
            \draw[Arete](0)--(1);
        \end{tikzpicture}
        \begin{tikzpicture}
            \node[Noeud](0)at(0,-1){};
            \node[Noeud](1)at(1,0){};
            \draw[Arete](1)--(0);
        \end{tikzpicture}
    }}
    \Tenseur
    \PP_{\scalebox{0.15}{
        \begin{tikzpicture}
            \node[Noeud](0)at(0,0){};
            \node[Noeud](1)at(1,-1){};
            \draw[Arete](0)--(1);
        \end{tikzpicture}
        \begin{tikzpicture}
            \node[Noeud](0)at(0,-1){};
            \node[Noeud](1)at(1,0){};
            \draw[Arete](1)--(0);
        \end{tikzpicture}
    }} \\
    & +
    \PP_{\scalebox{0.15}{
        \begin{tikzpicture}
            \node[Noeud](0)at(0,-1){};
            \node[Noeud](1)at(1,0){};
            \draw[Arete](1)--(0);
        \end{tikzpicture}
        \begin{tikzpicture}
            \node[Noeud](0)at(0,0){};
            \node[Noeud](1)at(1,-1){};
            \draw[Arete](0)--(1);
        \end{tikzpicture}
    }}
    \Tenseur
    \PP_{\scalebox{0.15}{
            \begin{tikzpicture}
            \node[Noeud](0)at(0,-1){};
            \node[Noeud](1)at(1,0){};
            \draw[Arete](1)--(0);
        \end{tikzpicture}
        \begin{tikzpicture}
            \node[Noeud](0)at(0,0){};
            \node[Noeud](1)at(1,-1){};
            \draw[Arete](0)--(1);
        \end{tikzpicture}
    }}
    +
    \PP_{\scalebox{0.15}{
        \begin{tikzpicture}
            \node[Noeud](0)at(0,-1){};
            \node[Noeud](1)at(1,0){};
            \draw[Arete](1)--(0);
            \node[Noeud](2)at(2,-1){};
            \draw[Arete](1)--(2);
        \end{tikzpicture}
        \begin{tikzpicture}
            \node[Noeud](0)at(0,-1){};
            \node[Noeud](1)at(1,-2){};
            \draw[Arete](0)--(1);
            \node[Noeud](2)at(2,0){};
            \draw[Arete](2)--(0);
        \end{tikzpicture}
    }}
    \Tenseur
    \PP_{\scalebox{0.15}{
        \begin{tikzpicture}
            \node[Noeud](0)at(0,0){};
        \end{tikzpicture}
        \begin{tikzpicture}
            \node[Noeud](0)at(0,0){};
        \end{tikzpicture}
    }}
    +
    \PP_{\scalebox{0.15}{
        \begin{tikzpicture}
            \node[Noeud](0)at(0,-1){};
            \node[Noeud](1)at(1,0){};
            \draw[Arete](1)--(0);
            \node[Noeud](2)at(2,-1){};
            \draw[Arete](1)--(2);
        \end{tikzpicture}
        \begin{tikzpicture}
            \node[Noeud](0)at(0,0){};
            \node[Noeud](1)at(1,-2){};
            \node[Noeud](2)at(2,-1){};
            \draw[Arete](2)--(1);
            \draw[Arete](0)--(2);
        \end{tikzpicture}
    }}
    \Tenseur
    \PP_{\scalebox{0.15}{
        \begin{tikzpicture}
            \node[Noeud](0)at(0,0){};
        \end{tikzpicture}
        \begin{tikzpicture}
            \node[Noeud](0)at(0,0){};
        \end{tikzpicture}
    }}
    +
    \PP_{\scalebox{0.15}{
        \begin{tikzpicture}
            \node[Noeud](0)at(0,-1){};
            \node[Noeud](1)at(1,0){};
            \draw[Arete](1)--(0);
            \node[Noeud](2)at(2,-2){};
            \node[Noeud](3)at(3,-1){};
            \draw[Arete](3)--(2);
            \draw[Arete](1)--(3);
        \end{tikzpicture}
        \begin{tikzpicture}
            \node[Noeud](0)at(0,-1){};
            \node[Noeud](1)at(1,-2){};
            \draw[Arete](0)--(1);
            \node[Noeud](2)at(2,0){};
            \draw[Arete](2)--(0);
            \node[Noeud](3)at(3,-1){};
            \draw[Arete](2)--(3);
        \end{tikzpicture}
    }}
    \Tenseur
    1.
\end{split}
\end{equation}

\begin{Remarque}
    It is well-known that the Hopf algebra $\PBT$~\cite{LR98, HNT05-2} is a
    Hopf subalgebra of $\FQSym$. Besides, we have the following sequence
    of injective Hopf maps:
    \begin{equation}
        \PBT \overset{\rho}{\hookrightarrow} \Baxter \hookrightarrow \FQSym.
    \end{equation}
    Indeed, by Proposition~\ref{prop:LienSylv}, every $\EquivS$-equivalence
    class is an union of some $\EquivBX$-equivalence classes. Denoting by
    $\{\PP_T\}_{T \in \EnsAB}$ the basis of $\PBT$ defined in accordance with
    (\ref{eq:EquivFQSym}) by the sylvester equivalence relation $\EquivS$, we have
    \begin{equation}
        \rho \left(\PP_T\right) = \sum_{\substack{T' \in \EnsAB \\ J := (T', T) \in \EnsABJ}} \PP_J.
    \end{equation}
    For example,
    \begin{align}
        \rho \left(\PP_{\scalebox{0.15}{
            \begin{tikzpicture}
                \node[Noeud,Marque1](0)at(0,-2){};
                \node[Noeud,Marque1](1)at(1,-1){};
                \draw[Arete](1)--(0);
                \node[Noeud,Marque1](2)at(2,0){};
                \draw[Arete](2)--(1);
                \node[Noeud,Marque1](3)at(3,-2){};
                \node[Noeud,Marque1](4)at(4,-1){};
                \draw[Arete](4)--(3);
                \draw[Arete](2)--(4);
            \end{tikzpicture}
        }}\right)
        & =
        \PP_{\scalebox{0.15}{
            \begin{tikzpicture}
                \node[Noeud](0)at(0,0){};
                \node[Noeud](1)at(1,-2){};
                \node[Noeud](2)at(2,-3){};
                \draw[Arete](1)--(2);
                \node[Noeud](3)at(3,-1){};
                \draw[Arete](3)--(1);
                \node[Noeud](4)at(4,-2){};
                \draw[Arete](3)--(4);
                \draw[Arete](0)--(3);
            \end{tikzpicture}
            \begin{tikzpicture}
                \node[Noeud,Marque1](0)at(0,-2){};
                \node[Noeud,Marque1](1)at(1,-1){};
                \draw[Arete](1)--(0);
                \node[Noeud,Marque1](2)at(2,0){};
                \draw[Arete](2)--(1);
                \node[Noeud,Marque1](3)at(3,-2){};
                \node[Noeud,Marque1](4)at(4,-1){};
                \draw[Arete](4)--(3);
                \draw[Arete](2)--(4);
            \end{tikzpicture}
        }}
        +
        \PP_{\scalebox{0.15}{
            \begin{tikzpicture}
                \node[Noeud](0)at(0,-1){};
                \node[Noeud](1)at(1,-2){};
                \node[Noeud](2)at(2,-3){};
                \draw[Arete](1)--(2);
                \draw[Arete](0)--(1);
                \node[Noeud](3)at(3,0){};
                \draw[Arete](3)--(0);
                \node[Noeud](4)at(4,-1){};
                \draw[Arete](3)--(4);
            \end{tikzpicture}
            \begin{tikzpicture}
                \node[Noeud,Marque1](0)at(0,-2){};
                \node[Noeud,Marque1](1)at(1,-1){};
                \draw[Arete](1)--(0);
                \node[Noeud,Marque1](2)at(2,0){};
                \draw[Arete](2)--(1);
                \node[Noeud,Marque1](3)at(3,-2){};
                \node[Noeud,Marque1](4)at(4,-1){};
                \draw[Arete](4)--(3);
                \draw[Arete](2)--(4);
            \end{tikzpicture}
        }}
        +
        \PP_{\scalebox{0.15}{
            \begin{tikzpicture}
                \node[Noeud](0)at(0,0){};
                \node[Noeud](1)at(1,-1){};
                \node[Noeud](2)at(2,-3){};
                \node[Noeud](3)at(3,-2){};
                \draw[Arete](3)--(2);
                \node[Noeud](4)at(4,-3){};
                \draw[Arete](3)--(4);
                \draw[Arete](1)--(3);
                \draw[Arete](0)--(1);
            \end{tikzpicture}
            \begin{tikzpicture}
                \node[Noeud,Marque1](0)at(0,-2){};
                \node[Noeud,Marque1](1)at(1,-1){};
                \draw[Arete](1)--(0);
                \node[Noeud,Marque1](2)at(2,0){};
                \draw[Arete](2)--(1);
                \node[Noeud,Marque1](3)at(3,-2){};
                \node[Noeud,Marque1](4)at(4,-1){};
                \draw[Arete](4)--(3);
                \draw[Arete](2)--(4);
            \end{tikzpicture}
        }}.
    \end{align}
\end{Remarque}

\subsection{Multiplicative bases}
Define the \emph{elementary} family $\left\{\E_J\right\}_{J \in \EnsABJ}$
and the \emph{homogeneous} family $\left\{\HH_J\right\}_{J \in \EnsABJ}$
respectively by:
\begin{equation}
    \E_J := \sum_{J \OrdBX J'} \PP_{J'} \hspace{2em} \text{and} \hspace{2em}
    \HH_J := \sum_{J' \OrdBX J} \PP_{J'}.
\end{equation}
These families are bases of $\Baxter$ since they are defined by triangularity.

Let $J_0 := (T^0_L, T^0_R)$ and $J_1 := (T^1_L, T^1_R)$ be two pairs of twin
binary trees. Let us define the pair of twin binary trees $J_0 \Over J_1$ by
$J_0 \Over J_1 := (T^0_L \Under T^1_L, T^0_R \Over T^1_R)$. In the same way,
the pair of twin binary trees $J_0 \Under J_1$ is defined by
$J_0 \Under J_1 := (T^0_L \Over T^1_L, T^0_R \Under T^1_R)$.

Using the multiplicative bases of $\FQSym$, we establish the following proposition:

\begin{Proposition} \label{prop:BasesMult}
    For all $J_0, J_1 \in \EnsABJ$, we have
    \begin{equation}
        \E_{J_0} \Prod \E_{J_1}   = \E_{J_0 \Over J_1} \hspace{2em} \text{and} \hspace{2em}
        \HH_{J_0} \Prod \HH_{J_1} = \HH_{J_0 \Under J_1}.
    \end{equation}
\end{Proposition}

\begin{Lemme} \label{lem:BaxterConnectee}
    Let $C$ be an equivalence class of $\EnsPermu_n /_\EquivBX$. The Baxter
    permutation belonging to $C$ is connected iff all the permutations of
    $C$ are connected.
\end{Lemme}

Let us say that a pair of twin binary trees $J$ is \emph{connected} if the
unique Baxter permutation $\sigma$ satisfying $\PSymb(\sigma) = J$ is connected.

\begin{Proposition}
    The Hopf algebra $\Baxter$ is free on the elements $\E_J$ where
    $J$ is a connected pair of twin binary trees.
\end{Proposition}

The generating series $B_C(z)$ of connected Baxter
permutations is $B_C(z) = 1 - B(z)^{-1}$. First dimensions of algebraic
generators of $\Baxter$ are $1$, $1$, $1$, $3$, $11$, $47$, $221$, $1113$,
$5903$, $32607$, $186143$, $1092015$.

\subsection{Bidendriform bialgebra structure}
A Hopf algebra $(H, \Prod, \Delta)$ can be fitted into a bidendriform
bialgebra structure~\cite{F05} if $(H^+, \Gauche, \Droite)$ is a dendriform algebra~\cite{Lod01}
and $(H^+, \DeltaG, \DeltaD)$ a codendriform coalgebra, where $H^+$ is the
augmentation ideal of $H$. The operators $\Gauche$, $\Droite$, $\DeltaG$
and $\DeltaD$ have to fulfil some compatibility relations. In particular,
for all $x, y \in H^+$, the product $\Prod$ of $H$ is retrieved by
$x \Prod y = x \Gauche y + x \Droite y$ and the coproduct $\Delta$ of $H$
is retrieved by $\Delta(x) = 1 \Tenseur x + \DeltaG(x) + \DeltaD(x) + x \Tenseur 1$.

The Hopf algebra $\FQSym$ admits a bidendriform bialgebra structure~\cite{F05}.
Indeed, for all $\sigma, \nu \in \EnsPermu$ set
\begin{align}
    \F_\sigma \Gauche \F_\nu & := \sum_{\substack{\pi \in \sigma \cshuffle \nu \\ \pi_{|\pi|} = \sigma_{|\sigma|}}} \F_\pi, \hspace{4em}
    \F_\sigma \Droite \F_\nu   := \sum_{\substack{\pi \in \sigma \cshuffle \nu \\ \pi_{|\pi|} = \nu_{|\nu|} + |\sigma|}} \F_\pi, \\
    \DeltaG(\F_\sigma)       & := \sum_{\substack{\sigma^{-1}_{|\sigma|} \leq i \leq |\sigma| - 1}} \F_{\Std(\sigma_1 \ldots \sigma_i)}
        \Tenseur \F_{\Std(\sigma_{i + 1} \ldots \sigma_{|\sigma|})}, \\
    \DeltaD(\F_\sigma)       & := \sum_{\substack{1 \leq i \leq \sigma^{-1}_{|\sigma|} - 1}} \F_{\Std(\sigma_1 \ldots \sigma_i)}
        \Tenseur \F_{\Std(\sigma_{i + 1} \ldots \sigma_{|\sigma|})}.
\end{align}

\begin{Proposition} \label{prop:MemeLettreEquiv}
    If $\equiv$ is an equivalence relation defined on $A^*$
    satisfying the conditions of Theorem~\ref{thm:HivertJanvier} and
    additionally, for all $u, v \in A^*$, the relation $u \equiv v$ implies
    $u_{|u|} = v_{|v|}$, then, the family defined in (\ref{eq:EquivFQSym})
    spans a bidendriform sub-bialgebra of $\FQSym$, and is free as an algebra,
    cofree as a coalgebra, self-dual, and the Lie algebra of its primitive
    elements is free.
\end{Proposition}

The equivalence relation $\EquivBX$ satisfies the premises of
Proposition~\ref{prop:MemeLettreEquiv} so that $\Baxter$ is free as an algebra,
cofree as a coalgebra, self-dual, and the Lie algebra of its primitive
elements is free.

\subsection{\texorpdfstring{The dual Hopf algebra $\Baxter^\star$}
                           {The dual Hopf algebra Baxter*}}
Let $\{\PP^\star_J\}_{J \in \EnsABJ}$ be the dual basis of the basis
$\{\PP_J\}_{J \in \EnsABJ}$. The Hopf algebra $\Baxter^\star$, dual of
$\Baxter$, is a quotient Hopf algebra of $\FQSym^\star$. More precisely,
\begin{equation}
    \Baxter^\star = \FQSym^\star / I
\end{equation}
where $I$ is the Hopf ideal of $\FQSym^\star$ spanned by the relations
$\F^\star_\sigma = \F^\star_\nu$ whenever $\sigma \EquivBX \nu$.

Let $\phi : \FQSym^\star \rightarrow \Baxter^\star$ be the canonical projection,
mapping $\F^\star_\sigma$ on $\PP^\star_J$ whenever $\PSymb(\sigma) = J$.
By definition, the product of $\Baxter^\star$ is
\begin{equation}
    \PP^\star_{J_0} \Prod \PP^\star_{J_1} = \phi \left(\F^\star_{\sigma} \Prod \F^\star_{\nu} \right)
\end{equation}
where $\sigma$ and $\nu$ are any permutations such that $\PSymb(\sigma) = J_0$
and $\PSymb(\nu) = J_1$.
For example,
\begin{equation}
\begin{split}
    \PP^\star_{\scalebox{0.15}{
        \begin{tikzpicture}
            \node[Noeud,Marque1](0)at(0,-1){};
            \node[Noeud,Marque1](1)at(1,-2){};
            \draw[Arete](0)--(1);
            \node[Noeud,Marque1](2)at(2,0){};
            \draw[Arete](2)--(0);
        \end{tikzpicture}
        \begin{tikzpicture}
            \node[Noeud](0)at(0,-1){};
            \node[Noeud](1)at(1,0){};
            \draw[Arete](1)--(0);
            \node[Noeud](2)at(2,-1){};
            \draw[Arete](1)--(2);
        \end{tikzpicture}
    }}
    \Prod
    \PP^\star_{\scalebox{0.15}{
        \begin{tikzpicture}
            \node[Noeud](0)at(0,0){};
            \node[Noeud](1)at(1,-1){};
            \draw[Arete](0)--(1);
        \end{tikzpicture}
        \begin{tikzpicture}
            \node[Noeud,Marque2](0)at(0,-1){};
            \node[Noeud,Marque2](1)at(1,0){};
            \draw[Arete](1)--(0);
        \end{tikzpicture}
    }}
    & =
    \PP^\star_{\scalebox{0.15}{
        \begin{tikzpicture}
            \node[Noeud,Marque1](0)at(0,-1){};
            \node[Noeud,Marque1](1)at(1,-2){};
            \draw[Arete](0)--(1);
            \node[Noeud,Marque1](2)at(2,0){};
            \draw[Arete](2)--(0);
            \node[Noeud](3)at(3,-1){};
            \node[Noeud](4)at(4,-2){};
            \draw[Arete](3)--(4);
            \draw[Arete](2)--(3);
        \end{tikzpicture}
        \begin{tikzpicture}
            \node[Noeud](0)at(0,-3){};
            \node[Noeud](1)at(1,-2){};
            \draw[Arete](1)--(0);
            \node[Noeud](2)at(2,-3){};
            \draw[Arete](1)--(2);
            \node[Noeud,Marque2](3)at(3,-1){};
            \draw[Arete](3)--(1);
            \node[Noeud,Marque2](4)at(4,0){};
            \draw[Arete](4)--(3);
        \end{tikzpicture}
    }}
    +
    \PP^\star_{\scalebox{0.15}{
        \begin{tikzpicture}
            \node[Noeud,Marque1](0)at(0,-1){};
            \node[Noeud,Marque1](1)at(1,-2){};
            \node[Noeud](2)at(2,-3){};
            \draw[Arete](1)--(2);
            \draw[Arete](0)--(1);
            \node[Noeud,Marque1](3)at(3,0){};
            \draw[Arete](3)--(0);
            \node[Noeud](4)at(4,-1){};
            \draw[Arete](3)--(4);
        \end{tikzpicture}
        \begin{tikzpicture}
            \node[Noeud](0)at(0,-3){};
            \node[Noeud](1)at(1,-2){};
            \draw[Arete](1)--(0);
            \node[Noeud,Marque2](2)at(2,-1){};
            \draw[Arete](2)--(1);
            \node[Noeud](3)at(3,-2){};
            \draw[Arete](2)--(3);
            \node[Noeud,Marque2](4)at(4,0){};
            \draw[Arete](4)--(2);
        \end{tikzpicture}
    }}
    +
    \PP^\star_{\scalebox{0.15}{
        \begin{tikzpicture}
            \node[Noeud,Marque1](0)at(0,-1){};
            \node[Noeud](1)at(1,-3){};
            \node[Noeud,Marque1](2)at(2,-2){};
            \draw[Arete](2)--(1);
            \draw[Arete](0)--(2);
            \node[Noeud,Marque1](3)at(3,0){};
            \draw[Arete](3)--(0);
            \node[Noeud](4)at(4,-1){};
            \draw[Arete](3)--(4);
        \end{tikzpicture}
        \begin{tikzpicture}
            \node[Noeud](0)at(0,-2){};
            \node[Noeud,Marque2](1)at(1,-1){};
            \draw[Arete](1)--(0);
            \node[Noeud](2)at(2,-2){};
            \node[Noeud](3)at(3,-3){};
            \draw[Arete](2)--(3);
            \draw[Arete](1)--(2);
            \node[Noeud,Marque2](4)at(4,0){};
            \draw[Arete](4)--(1);
        \end{tikzpicture}
    }}
    +
    \PP^\star_{\scalebox{0.15}{
        \begin{tikzpicture}
            \node[Noeud](0)at(0,-2){};
            \node[Noeud,Marque1](1)at(1,-1){};
            \draw[Arete](1)--(0);
            \node[Noeud,Marque1](2)at(2,-2){};
            \draw[Arete](1)--(2);
            \node[Noeud,Marque1](3)at(3,0){};
            \draw[Arete](3)--(1);
            \node[Noeud](4)at(4,-1){};
            \draw[Arete](3)--(4);
        \end{tikzpicture}
        \begin{tikzpicture}
            \node[Noeud,Marque2](0)at(0,-1){};
            \node[Noeud](1)at(1,-3){};
            \node[Noeud](2)at(2,-2){};
            \draw[Arete](2)--(1);
            \node[Noeud](3)at(3,-3){};
            \draw[Arete](2)--(3);
            \draw[Arete](0)--(2);
            \node[Noeud,Marque2](4)at(4,0){};
            \draw[Arete](4)--(0);
        \end{tikzpicture}
    }}
    +
    \PP^\star_{\scalebox{0.15}{
        \begin{tikzpicture}
            \node[Noeud,Marque1](0)at(0,-1){};
            \node[Noeud,Marque1](1)at(1,-2){};
            \node[Noeud](2)at(2,-3){};
            \node[Noeud](3)at(3,-4){};
            \draw[Arete](2)--(3);
            \draw[Arete](1)--(2);
            \draw[Arete](0)--(1);
            \node[Noeud,Marque1](4)at(4,0){};
            \draw[Arete](4)--(0);
        \end{tikzpicture}
        \begin{tikzpicture}
            \node[Noeud](0)at(0,-3){};
            \node[Noeud](1)at(1,-2){};
            \draw[Arete](1)--(0);
            \node[Noeud,Marque2](2)at(2,-1){};
            \draw[Arete](2)--(1);
            \node[Noeud,Marque2](3)at(3,0){};
            \draw[Arete](3)--(2);
            \node[Noeud](4)at(4,-1){};
            \draw[Arete](3)--(4);
        \end{tikzpicture}
    }} \\
    & +
    \PP^\star_{\scalebox{0.15}{
        \begin{tikzpicture}
            \node[Noeud,Marque1](0)at(0,-1){};
            \node[Noeud](1)at(1,-3){};
            \node[Noeud,Marque1](2)at(2,-2){};
            \draw[Arete](2)--(1);
            \node[Noeud](3)at(3,-3){};
            \draw[Arete](2)--(3);
            \draw[Arete](0)--(2);
            \node[Noeud,Marque1](4)at(4,0){};
            \draw[Arete](4)--(0);
        \end{tikzpicture}
        \begin{tikzpicture}
            \node[Noeud](0)at(0,-2){};
            \node[Noeud,Marque2](1)at(1,-1){};
            \draw[Arete](1)--(0);
            \node[Noeud](2)at(2,-2){};
            \draw[Arete](1)--(2);
            \node[Noeud,Marque2](3)at(3,0){};
            \draw[Arete](3)--(1);
            \node[Noeud](4)at(4,-1){};
            \draw[Arete](3)--(4);
        \end{tikzpicture}
    }}
    +
    \PP^\star_{\scalebox{0.15}{
        \begin{tikzpicture}
            \node[Noeud,Marque1](0)at(0,-1){};
            \node[Noeud](1)at(1,-3){};
            \node[Noeud](2)at(2,-4){};
            \draw[Arete](1)--(2);
            \node[Noeud,Marque1](3)at(3,-2){};
            \draw[Arete](3)--(1);
            \draw[Arete](0)--(3);
            \node[Noeud,Marque1](4)at(4,0){};
            \draw[Arete](4)--(0);
        \end{tikzpicture}
        \begin{tikzpicture}
            \node[Noeud](0)at(0,-2){};
            \node[Noeud,Marque2](1)at(1,-1){};
            \draw[Arete](1)--(0);
            \node[Noeud,Marque2](2)at(2,0){};
            \draw[Arete](2)--(1);
            \node[Noeud](3)at(3,-1){};
            \node[Noeud](4)at(4,-2){};
            \draw[Arete](3)--(4);
            \draw[Arete](2)--(3);
        \end{tikzpicture}
    }}
    +
    \PP^\star_{\scalebox{0.15}{
        \begin{tikzpicture}
            \node[Noeud](0)at(0,-2){};
            \node[Noeud,Marque1](1)at(1,-1){};
            \draw[Arete](1)--(0);
            \node[Noeud](2)at(2,-3){};
            \node[Noeud,Marque1](3)at(3,-2){};
            \draw[Arete](3)--(2);
            \draw[Arete](1)--(3);
            \node[Noeud,Marque1](4)at(4,0){};
            \draw[Arete](4)--(1);
        \end{tikzpicture}
        \begin{tikzpicture}
            \node[Noeud,Marque2](0)at(0,-1){};
            \node[Noeud](1)at(1,-2){};
            \draw[Arete](0)--(1);
            \node[Noeud,Marque2](2)at(2,0){};
            \draw[Arete](2)--(0);
            \node[Noeud](3)at(3,-1){};
            \node[Noeud](4)at(4,-2){};
            \draw[Arete](3)--(4);
            \draw[Arete](2)--(3);
        \end{tikzpicture}
    }}
    +
    \PP^\star_{\scalebox{0.15}{
        \begin{tikzpicture}
            \node[Noeud](0)at(0,-2){};
            \node[Noeud,Marque1](1)at(1,-1){};
            \draw[Arete](1)--(0);
            \node[Noeud,Marque1](2)at(2,-2){};
            \node[Noeud](3)at(3,-3){};
            \draw[Arete](2)--(3);
            \draw[Arete](1)--(2);
            \node[Noeud,Marque1](4)at(4,0){};
            \draw[Arete](4)--(1);
        \end{tikzpicture}
        \begin{tikzpicture}
            \node[Noeud,Marque2](0)at(0,-1){};
            \node[Noeud](1)at(1,-3){};
            \node[Noeud](2)at(2,-2){};
            \draw[Arete](2)--(1);
            \draw[Arete](0)--(2);
            \node[Noeud,Marque2](3)at(3,0){};
            \draw[Arete](3)--(0);
            \node[Noeud](4)at(4,-1){};
            \draw[Arete](3)--(4);
        \end{tikzpicture}
    }}
    +
    \PP^\star_{\scalebox{0.15}{
        \begin{tikzpicture}
            \node[Noeud](0)at(0,-2){};
            \node[Noeud](1)at(1,-3){};
            \draw[Arete](0)--(1);
            \node[Noeud,Marque1](2)at(2,-1){};
            \draw[Arete](2)--(0);
            \node[Noeud,Marque1](3)at(3,-2){};
            \draw[Arete](2)--(3);
            \node[Noeud,Marque1](4)at(4,0){};
            \draw[Arete](4)--(2);
        \end{tikzpicture}
        \begin{tikzpicture}
            \node[Noeud,Marque2](0)at(0,-1){};
            \node[Noeud,Marque2](1)at(1,0){};
            \draw[Arete](1)--(0);
            \node[Noeud](2)at(2,-2){};
            \node[Noeud](3)at(3,-1){};
            \draw[Arete](3)--(2);
            \node[Noeud](4)at(4,-2){};
            \draw[Arete](3)--(4);
            \draw[Arete](1)--(3);
        \end{tikzpicture}
    }}.
\end{split}
\end{equation}

In the same way, the coproduct of $\Baxter^\star$ is
\begin{equation}
    \Delta(\PP_J) = (\phi \Tenseur \phi)(\Delta \left(\F^\star_\sigma \right))
\end{equation}
where $\sigma$ is any permutation such that $\PSymb(\sigma) = J$.
For example,
\begin{equation}
    \Delta
    \PP^\star_{\scalebox{0.15}{
        \begin{tikzpicture}
            \node[Noeud](0)at(0,-1){};
            \node[Noeud](1)at(1,0){};
            \draw[Arete](1)--(0);
            \node[Noeud](2)at(2,-2){};
            \node[Noeud](3)at(3,-1){};
            \draw[Arete](3)--(2);
            \draw[Arete](1)--(3);
        \end{tikzpicture}
        \begin{tikzpicture}
            \node[Noeud](0)at(0,-1){};
            \node[Noeud](1)at(1,-2){};
            \draw[Arete](0)--(1);
            \node[Noeud](2)at(2,0){};
            \draw[Arete](2)--(0);
            \node[Noeud](3)at(3,-1){};
            \draw[Arete](2)--(3);
        \end{tikzpicture}
    }}
    =
    1
    \Tenseur
    \PP^\star_{\scalebox{0.15}{
        \begin{tikzpicture}
            \node[Noeud](0)at(0,-1){};
            \node[Noeud](1)at(1,0){};
            \draw[Arete](1)--(0);
            \node[Noeud](2)at(2,-2){};
            \node[Noeud](3)at(3,-1){};
            \draw[Arete](3)--(2);
            \draw[Arete](1)--(3);
        \end{tikzpicture}
        \begin{tikzpicture}
            \node[Noeud](0)at(0,-1){};
            \node[Noeud](1)at(1,-2){};
            \draw[Arete](0)--(1);
            \node[Noeud](2)at(2,0){};
            \draw[Arete](2)--(0);
            \node[Noeud](3)at(3,-1){};
            \draw[Arete](2)--(3);
        \end{tikzpicture}
    }}
    +
    \PP^\star_{\scalebox{0.15}{
        \begin{tikzpicture}
            \node[Noeud](0)at(0,0){};
        \end{tikzpicture}
        \begin{tikzpicture}
            \node[Noeud](0)at(0,0){};
        \end{tikzpicture}
    }}
    \Tenseur
    \PP^\star_{\scalebox{0.15}{
        \begin{tikzpicture}
            \node[Noeud](0)at(0,0){};
            \node[Noeud](1)at(1,-2){};
            \node[Noeud](2)at(2,-1){};
            \draw[Arete](2)--(1);
            \draw[Arete](0)--(2);
        \end{tikzpicture}
        \begin{tikzpicture}
            \node[Noeud](0)at(0,-1){};
            \node[Noeud](1)at(1,0){};
            \draw[Arete](1)--(0);
            \node[Noeud](2)at(2,-1){};
            \draw[Arete](1)--(2);
        \end{tikzpicture}
    }}
    +
    \PP^\star_{\scalebox{0.15}{
        \begin{tikzpicture}
            \node[Noeud](0)at(0,-1){};
            \node[Noeud](1)at(1,0){};
            \draw[Arete](1)--(0);
        \end{tikzpicture}
        \begin{tikzpicture}
            \node[Noeud](0)at(0,0){};
            \node[Noeud](1)at(1,-1){};
            \draw[Arete](0)--(1);
        \end{tikzpicture}
    }}
    \Tenseur
    \PP^\star_{\scalebox{0.15}{
        \begin{tikzpicture}
            \node[Noeud](0)at(0,-1){};
            \node[Noeud](1)at(1,0){};
            \draw[Arete](1)--(0);
        \end{tikzpicture}
        \begin{tikzpicture}
            \node[Noeud](0)at(0,0){};
            \node[Noeud](1)at(1,-1){};
            \draw[Arete](0)--(1);
        \end{tikzpicture}
    }}
    +
    \PP^\star_{\scalebox{0.15}{
        \begin{tikzpicture}
            \node[Noeud](0)at(0,-1){};
            \node[Noeud](1)at(1,0){};
            \draw[Arete](1)--(0);
            \node[Noeud](2)at(2,-1){};
            \draw[Arete](1)--(2);
        \end{tikzpicture}
        \begin{tikzpicture}
            \node[Noeud](0)at(0,-1){};
            \node[Noeud](1)at(1,-2){};
            \draw[Arete](0)--(1);
            \node[Noeud](2)at(2,0){};
            \draw[Arete](2)--(0);
        \end{tikzpicture}
    }}
    \Tenseur
    \PP^\star_{\scalebox{0.15}{
        \begin{tikzpicture}
            \node[Noeud](0)at(0,0){};
        \end{tikzpicture}
        \begin{tikzpicture}
            \node[Noeud](0)at(0,0){};
        \end{tikzpicture}
    }}
    +
    \PP^\star_{\scalebox{0.15}{
        \begin{tikzpicture}
            \node[Noeud](0)at(0,-1){};
            \node[Noeud](1)at(1,0){};
            \draw[Arete](1)--(0);
            \node[Noeud](2)at(2,-2){};
            \node[Noeud](3)at(3,-1){};
            \draw[Arete](3)--(2);
            \draw[Arete](1)--(3);
        \end{tikzpicture}
        \begin{tikzpicture}
            \node[Noeud](0)at(0,-1){};
            \node[Noeud](1)at(1,-2){};
            \draw[Arete](0)--(1);
            \node[Noeud](2)at(2,0){};
            \draw[Arete](2)--(0);
            \node[Noeud](3)at(3,-1){};
            \draw[Arete](2)--(3);
        \end{tikzpicture}
    }}
    \Tenseur
    1.
\end{equation}

\begin{Remarque}
    By Proposition~\ref{prop:MemeLettreEquiv}, the Hopf algebras $\Baxter$
    and $\Baxter^\star$ are isomorphic. However, denoting by
    $\theta : \Baxter \hookrightarrow \FQSym$ the injection from $\Baxter$
    to $\FQSym$, $\psi : \FQSym \leftrightarrow \FQSym^\star$ the isomorphism
    from $\FQSym$ to $\FQSym^\star$ defined by $\psi(\F_\sigma) := \F^\star_{\sigma^{-1}}$,
    and $\phi : \FQSym^\star \twoheadrightarrow \Baxter^\star$ the surjection
    from $\FQSym^\star$ to $\Baxter^\star$,
    the map $\phi \circ \psi \circ \theta : \Baxter \rightarrow \Baxter^\star$ is not
    an isomorphism. Indeed:
    \begin{align}
        \phi \circ \psi \circ \theta
        \PP_{\scalebox{0.15}{
            \begin{tikzpicture}
                \node[Noeud](0)at(0,-1){};
                \node[Noeud](1)at(1,0){};
                \draw[Arete](1)--(0);
                \node[Noeud](2)at(2,-2){};
                \node[Noeud](3)at(3,-1){};
                \draw[Arete](3)--(2);
                \draw[Arete](1)--(3);
            \end{tikzpicture}
            \begin{tikzpicture}
                \node[Noeud](0)at(0,-1){};
                \node[Noeud](1)at(1,-2){};
                \draw[Arete](0)--(1);
                \node[Noeud](2)at(2,0){};
                \draw[Arete](2)--(0);
                \node[Noeud](3)at(3,-1){};
                \draw[Arete](2)--(3);
            \end{tikzpicture}
        }}
        & = \phi \circ \psi \left( \F_{2143} + \F_{2413} \right)
          = \phi \left( \F^\star_{2143} + \F^\star_{3142} \right)
          =
        \PP^\star_{\scalebox{0.15}{
            \begin{tikzpicture}
                \node[Noeud](0)at(0,-1){};
                \node[Noeud](1)at(1,0){};
                \draw[Arete](1)--(0);
                \node[Noeud](2)at(2,-2){};
                \node[Noeud](3)at(3,-1){};
                \draw[Arete](3)--(2);
                \draw[Arete](1)--(3);
            \end{tikzpicture}
            \begin{tikzpicture}
                \node[Noeud](0)at(0,-1){};
                \node[Noeud](1)at(1,-2){};
                \draw[Arete](0)--(1);
                \node[Noeud](2)at(2,0){};
                \draw[Arete](2)--(0);
                \node[Noeud](3)at(3,-1){};
                \draw[Arete](2)--(3);
            \end{tikzpicture}
        }}
        +
        \PP^\star_{\scalebox{0.15}{
            \begin{tikzpicture}
                \node[Noeud](0)at(0,-1){};
                \node[Noeud](1)at(1,-2){};
                \draw[Arete](0)--(1);
                \node[Noeud](2)at(2,0){};
                \draw[Arete](2)--(0);
                \node[Noeud](3)at(3,-1){};
                \draw[Arete](2)--(3);
            \end{tikzpicture}
            \begin{tikzpicture}
                \node[Noeud](0)at(0,-1){};
                \node[Noeud](1)at(1,0){};
                \draw[Arete](1)--(0);
                \node[Noeud](2)at(2,-2){};
                \node[Noeud](3)at(3,-1){};
                \draw[Arete](3)--(2);
                \draw[Arete](1)--(3);
            \end{tikzpicture}
        }}, \\
        \phi \circ \psi \circ \theta
        \PP_{\scalebox{0.15}{
            \begin{tikzpicture}
                \node[Noeud](0)at(0,-1){};
                \node[Noeud](1)at(1,-2){};
                \draw[Arete](0)--(1);
                \node[Noeud](2)at(2,0){};
                \draw[Arete](2)--(0);
                \node[Noeud](3)at(3,-1){};
                \draw[Arete](2)--(3);
            \end{tikzpicture}
            \begin{tikzpicture}
                \node[Noeud](0)at(0,-1){};
                \node[Noeud](1)at(1,0){};
                \draw[Arete](1)--(0);
                \node[Noeud](2)at(2,-2){};
                \node[Noeud](3)at(3,-1){};
                \draw[Arete](3)--(2);
                \draw[Arete](1)--(3);
            \end{tikzpicture}
        }}
        & = \phi \circ \psi \left( \F_{3142} + \F_{3412} \right)
          = \phi \left( \F^\star_{2413} + \F^\star_{3412} \right)
          =
        \PP^\star_{\scalebox{0.15}{
            \begin{tikzpicture}
                \node[Noeud](0)at(0,-1){};
                \node[Noeud](1)at(1,0){};
                \draw[Arete](1)--(0);
                \node[Noeud](2)at(2,-2){};
                \node[Noeud](3)at(3,-1){};
                \draw[Arete](3)--(2);
                \draw[Arete](1)--(3);
            \end{tikzpicture}
            \begin{tikzpicture}
                \node[Noeud](0)at(0,-1){};
                \node[Noeud](1)at(1,-2){};
                \draw[Arete](0)--(1);
                \node[Noeud](2)at(2,0){};
                \draw[Arete](2)--(0);
                \node[Noeud](3)at(3,-1){};
                \draw[Arete](2)--(3);
            \end{tikzpicture}
        }}
        +
        \PP^\star_{\scalebox{0.15}{
            \begin{tikzpicture}
                \node[Noeud](0)at(0,-1){};
                \node[Noeud](1)at(1,-2){};
                \draw[Arete](0)--(1);
                \node[Noeud](2)at(2,0){};
                \draw[Arete](2)--(0);
                \node[Noeud](3)at(3,-1){};
                \draw[Arete](2)--(3);
            \end{tikzpicture}
            \begin{tikzpicture}
                \node[Noeud](0)at(0,-1){};
                \node[Noeud](1)at(1,0){};
                \draw[Arete](1)--(0);
                \node[Noeud](2)at(2,-2){};
                \node[Noeud](3)at(3,-1){};
                \draw[Arete](3)--(2);
                \draw[Arete](1)--(3);
            \end{tikzpicture}
        }},
    \end{align}
    showing that $\phi \circ \psi \circ \theta$ is not injective.
\end{Remarque}

\subsection{Primitive and totally primitive elements}

\subsubsection{Primitive elements}
Since the family $\{\E_J\}_{J \in C}$, where $C$ is the set of connected pairs of twin binary
trees, are indecomposable elements of $\Baxter$, its dual family $\{\E^\star_J\}_{J \in C}$
forms a basis of the Lie algebra $\mathfrak{p}^\star$ of the primitive elements of $\Baxter^\star$.
By Proposition~\ref{prop:MemeLettreEquiv}, the Lie algebra $\mathfrak{p}^\star$ is free.

\subsubsection{Totally primitive elements}
An element $x$ of a bidendriform bialgebra is \emph{totally primitive}
if $\DeltaG(x) = 0 = \DeltaD(x)$.

Following~\cite{F05}, the generating series $B_T(z)$ of the totally
primitive elements of $\Baxter$ is $B_T(z) = \frac{B(z) - 1}{B(z)^2}$. First
dimensions of totally primitive  elements of $\Baxter$ are $0$, $1$, $0$,
$1$, $4$, $19$, $96$, $511$, $2832$, $16215$, $95374$, $573837$. Here follows
a basis of the totally primitive elements of $\Baxter$ of order $1$, $3$ and~$4$:
\begin{align}
    t_{1, 1} & =
    \PP_{\scalebox{0.15}{
        \begin{tikzpicture}
            \node[Noeud](0)at(0,0){};
        \end{tikzpicture}
        \begin{tikzpicture}
            \node[Noeud](0)at(0,0){};
        \end{tikzpicture}
    }}, \\[1em]
    t_{3, 1} & =
    \PP_{\scalebox{0.15}{
        \begin{tikzpicture}
            \node[Noeud](0)at(0,-1){};
            \node[Noeud](1)at(1,0){};
            \draw[Arete](1)--(0);
            \node[Noeud](2)at(2,-1){};
            \draw[Arete](1)--(2);
        \end{tikzpicture}
        \begin{tikzpicture}
            \node[Noeud](0)at(0,0){};
            \node[Noeud](1)at(1,-2){};
            \node[Noeud](2)at(2,-1){};
            \draw[Arete](2)--(1);
            \draw[Arete](0)--(2);
        \end{tikzpicture}
    }}
    -
    \PP_{\scalebox{0.15}{
        \begin{tikzpicture}
            \node[Noeud](0)at(0,0){};
            \node[Noeud](1)at(1,-2){};
            \node[Noeud](2)at(2,-1){};
            \draw[Arete](2)--(1);
            \draw[Arete](0)--(2);
        \end{tikzpicture}
        \begin{tikzpicture}
            \node[Noeud](0)at(0,-1){};
            \node[Noeud](1)at(1,0){};
            \draw[Arete](1)--(0);
            \node[Noeud](2)at(2,-1){};
            \draw[Arete](1)--(2);
        \end{tikzpicture}
    }}, \\[1em]
    t_{4, 1} & =
    \PP_{\scalebox{0.15}{
        \begin{tikzpicture}
            \node[Noeud](0)at(0,-2){};
            \node[Noeud](1)at(1,-1){};
            \draw[Arete](1)--(0);
            \node[Noeud](2)at(2,0){};
            \draw[Arete](2)--(1);
            \node[Noeud](3)at(3,-1){};
            \draw[Arete](2)--(3);
        \end{tikzpicture}
        \begin{tikzpicture}
            \node[Noeud](0)at(0,0){};
            \node[Noeud](1)at(1,-1){};
            \node[Noeud](2)at(2,-3){};
            \node[Noeud](3)at(3,-2){};
            \draw[Arete](3)--(2);
            \draw[Arete](1)--(3);
            \draw[Arete](0)--(1);
        \end{tikzpicture}
    }}
    +
    \PP_{\scalebox{0.15}{
        \begin{tikzpicture}
            \node[Noeud](0)at(0,-2){};
            \node[Noeud](1)at(1,-1){};
            \draw[Arete](1)--(0);
            \node[Noeud](2)at(2,0){};
            \draw[Arete](2)--(1);
            \node[Noeud](3)at(3,-1){};
            \draw[Arete](2)--(3);
        \end{tikzpicture}
        \begin{tikzpicture}
            \node[Noeud](0)at(0,0){};
            \node[Noeud](1)at(1,-2){};
            \node[Noeud](2)at(2,-3){};
            \draw[Arete](1)--(2);
            \node[Noeud](3)at(3,-1){};
            \draw[Arete](3)--(1);
            \draw[Arete](0)--(3);
        \end{tikzpicture}
    }}
    +
    \PP_{\scalebox{0.15}{
        \begin{tikzpicture}
            \node[Noeud](0)at(0,0){};
            \node[Noeud](1)at(1,-2){};
            \node[Noeud](2)at(2,-3){};
            \draw[Arete](1)--(2);
            \node[Noeud](3)at(3,-1){};
            \draw[Arete](3)--(1);
            \draw[Arete](0)--(3);
        \end{tikzpicture}
        \begin{tikzpicture}
            \node[Noeud](0)at(0,-2){};
            \node[Noeud](1)at(1,-1){};
            \draw[Arete](1)--(0);
            \node[Noeud](2)at(2,0){};
            \draw[Arete](2)--(1);
            \node[Noeud](3)at(3,-1){};
            \draw[Arete](2)--(3);
        \end{tikzpicture}
    }}
    +
    \PP_{\scalebox{0.15}{
        \begin{tikzpicture}
            \node[Noeud](0)at(0,0){};
            \node[Noeud](1)at(1,-1){};
            \node[Noeud](2)at(2,-3){};
            \node[Noeud](3)at(3,-2){};
            \draw[Arete](3)--(2);
            \draw[Arete](1)--(3);
            \draw[Arete](0)--(1);
        \end{tikzpicture}
        \begin{tikzpicture}
            \node[Noeud](0)at(0,-2){};
            \node[Noeud](1)at(1,-1){};
            \draw[Arete](1)--(0);
            \node[Noeud](2)at(2,0){};
            \draw[Arete](2)--(1);
            \node[Noeud](3)at(3,-1){};
            \draw[Arete](2)--(3);
        \end{tikzpicture}
    }} \\
    & -
    \PP_{\scalebox{0.15}{
        \begin{tikzpicture}
            \node[Noeud](0)at(0,-1){};
            \node[Noeud](1)at(1,-2){};
            \draw[Arete](0)--(1);
            \node[Noeud](2)at(2,0){};
            \draw[Arete](2)--(0);
            \node[Noeud](3)at(3,-1){};
            \draw[Arete](2)--(3);
        \end{tikzpicture}
        \begin{tikzpicture}
            \node[Noeud](0)at(0,-1){};
            \node[Noeud](1)at(1,0){};
            \draw[Arete](1)--(0);
            \node[Noeud](2)at(2,-2){};
            \node[Noeud](3)at(3,-1){};
            \draw[Arete](3)--(2);
            \draw[Arete](1)--(3);
        \end{tikzpicture}
    }}
    -
    \PP_{\scalebox{0.15}{
        \begin{tikzpicture}
            \node[Noeud](0)at(0,0){};
            \node[Noeud](1)at(1,-3){};
            \node[Noeud](2)at(2,-2){};
            \draw[Arete](2)--(1);
            \node[Noeud](3)at(3,-1){};
            \draw[Arete](3)--(2);
            \draw[Arete](0)--(3);
        \end{tikzpicture}
        \begin{tikzpicture}
            \node[Noeud](0)at(0,-1){};
            \node[Noeud](1)at(1,0){};
            \draw[Arete](1)--(0);
            \node[Noeud](2)at(2,-1){};
            \node[Noeud](3)at(3,-2){};
            \draw[Arete](2)--(3);
            \draw[Arete](1)--(2);
        \end{tikzpicture}
    }}
    -
    \PP_{\scalebox{0.15}{
        \begin{tikzpicture}
            \node[Noeud](0)at(0,0){};
            \node[Noeud](1)at(1,-2){};
            \node[Noeud](2)at(2,-1){};
            \draw[Arete](2)--(1);
            \node[Noeud](3)at(3,-2){};
            \draw[Arete](2)--(3);
            \draw[Arete](0)--(2);
        \end{tikzpicture}
        \begin{tikzpicture}
            \node[Noeud](0)at(0,-1){};
            \node[Noeud](1)at(1,0){};
            \draw[Arete](1)--(0);
            \node[Noeud](2)at(2,-2){};
            \node[Noeud](3)at(3,-1){};
            \draw[Arete](3)--(2);
            \draw[Arete](1)--(3);
        \end{tikzpicture}
    }}, \nonumber \\
    t_{4, 2} & =
    \PP_{\scalebox{0.15}{
        \begin{tikzpicture}
            \node[Noeud](0)at(0,-1){};
            \node[Noeud](1)at(1,0){};
            \draw[Arete](1)--(0);
            \node[Noeud](2)at(2,-2){};
            \node[Noeud](3)at(3,-1){};
            \draw[Arete](3)--(2);
            \draw[Arete](1)--(3);
        \end{tikzpicture}
        \begin{tikzpicture}
            \node[Noeud](0)at(0,0){};
            \node[Noeud](1)at(1,-2){};
            \node[Noeud](2)at(2,-1){};
            \draw[Arete](2)--(1);
            \node[Noeud](3)at(3,-2){};
            \draw[Arete](2)--(3);
            \draw[Arete](0)--(2);
        \end{tikzpicture}
    }}
    -
    \PP_{\scalebox{0.15}{
        \begin{tikzpicture}
            \node[Noeud](0)at(0,0){};
            \node[Noeud](1)at(1,-3){};
            \node[Noeud](2)at(2,-2){};
            \draw[Arete](2)--(1);
            \node[Noeud](3)at(3,-1){};
            \draw[Arete](3)--(2);
            \draw[Arete](0)--(3);
        \end{tikzpicture}
        \begin{tikzpicture}
            \node[Noeud](0)at(0,-1){};
            \node[Noeud](1)at(1,0){};
            \draw[Arete](1)--(0);
            \node[Noeud](2)at(2,-1){};
            \node[Noeud](3)at(3,-2){};
            \draw[Arete](2)--(3);
            \draw[Arete](1)--(2);
        \end{tikzpicture}
    }}, \\
    t_{4, 3} & =
    \PP_{\scalebox{0.15}{
        \begin{tikzpicture}
            \node[Noeud](0)at(0,-1){};
            \node[Noeud](1)at(1,0){};
            \draw[Arete](1)--(0);
            \node[Noeud](2)at(2,-1){};
            \node[Noeud](3)at(3,-2){};
            \draw[Arete](2)--(3);
            \draw[Arete](1)--(2);
        \end{tikzpicture}
        \begin{tikzpicture}
            \node[Noeud](0)at(0,0){};
            \node[Noeud](1)at(1,-3){};
            \node[Noeud](2)at(2,-2){};
            \draw[Arete](2)--(1);
            \node[Noeud](3)at(3,-1){};
            \draw[Arete](3)--(2);
            \draw[Arete](0)--(3);
        \end{tikzpicture}
    }}
    -
    \PP_{\scalebox{0.15}{
        \begin{tikzpicture}
            \node[Noeud](0)at(0,0){};
            \node[Noeud](1)at(1,-2){};
            \node[Noeud](2)at(2,-1){};
            \draw[Arete](2)--(1);
            \node[Noeud](3)at(3,-2){};
            \draw[Arete](2)--(3);
            \draw[Arete](0)--(2);
        \end{tikzpicture}
        \begin{tikzpicture}
            \node[Noeud](0)at(0,-1){};
            \node[Noeud](1)at(1,0){};
            \draw[Arete](1)--(0);
            \node[Noeud](2)at(2,-2){};
            \node[Noeud](3)at(3,-1){};
            \draw[Arete](3)--(2);
            \draw[Arete](1)--(3);
        \end{tikzpicture}
    }}, \\
    t_{4, 4} & =
    \PP_{\scalebox{0.15}{
        \begin{tikzpicture}
            \node[Noeud](0)at(0,-1){};
            \node[Noeud](1)at(1,0){};
            \draw[Arete](1)--(0);
            \node[Noeud](2)at(2,-2){};
            \node[Noeud](3)at(3,-1){};
            \draw[Arete](3)--(2);
            \draw[Arete](1)--(3);
        \end{tikzpicture}
        \begin{tikzpicture}
            \node[Noeud](0)at(0,-1){};
            \node[Noeud](1)at(1,-2){};
            \draw[Arete](0)--(1);
            \node[Noeud](2)at(2,0){};
            \draw[Arete](2)--(0);
            \node[Noeud](3)at(3,-1){};
            \draw[Arete](2)--(3);
        \end{tikzpicture}
    }}
    -
    \PP_{\scalebox{0.15}{
        \begin{tikzpicture}
            \node[Noeud](0)at(0,-1){};
            \node[Noeud](1)at(1,-2){};
            \draw[Arete](0)--(1);
            \node[Noeud](2)at(2,0){};
            \draw[Arete](2)--(0);
            \node[Noeud](3)at(3,-1){};
            \draw[Arete](2)--(3);
        \end{tikzpicture}
        \begin{tikzpicture}
            \node[Noeud](0)at(0,-1){};
            \node[Noeud](1)at(1,0){};
            \draw[Arete](1)--(0);
            \node[Noeud](2)at(2,-2){};
            \node[Noeud](3)at(3,-1){};
            \draw[Arete](3)--(2);
            \draw[Arete](1)--(3);
        \end{tikzpicture}
    }}.
\end{align}

$\Baxter$ is free as dendriform algebra on its totally primitive elements.

\bibliographystyle{plain}
\bibliography{Bibliographie}

\end{document}